\documentclass[12pt,twoside,a4paper]{article}

\usepackage{amssymb}
\usepackage{amscd}
\usepackage{amsthm}
\usepackage{amsxtra}
\usepackage{amsmath}
\usepackage{latexsym}
\usepackage{graphicx}
\usepackage{delarray}%\usepackage{epic}
\usepackage{float}
\usepackage[dvips]{epsfig}
\usepackage{subfig}
\usepackage{dsfont}
\usepackage{hyperref
}\usepackage{breakurl}

\usepackage{setspace}

\begin{document}
%\doublespacing
%%%%%%%%%%%%%%%%%%%%%%%%%%%%%%%%%%%%%%%%%%%%%%%%%%%%%%%%%%%%

\theoremstyle{plain}
\newtheorem{theorem}{Theorem}[section]
\newtheorem{claim}{Claim}%[section]
\newtheorem{lemma}[theorem]{Lemma}
\newtheorem{proposition}[theorem]{Proposition}
\newtheorem{corollary}[theorem]{Corollary}
\newtheorem{remark}{Remark}

\theoremstyle{remark}
\newtheorem{example}[theorem]{Example} 	
\theoremstyle{definition}
\newtheorem{definition}{Definition}
\hfuzz5pt % Don't bother to report over-full boxes if over-edge is < 5pt

%%%%%%%%%%%%%%%%%%%%%%%%%%%%%%%%%%%%%%%%%%%%%%%%%%%%%%%%%%%%

\newcommand{\gt}{\tilde{g}}
\newcommand{\R}{\mathbb{R}}
\newcommand{\Z}{\mathbb{Z}}
\newcommand{\T}{\mathbb{T}}
\newcommand{\N}{\mathbb{N}}
\newcommand{\Zt}{\mathbb{Z}^2}
\newcommand{\Zd}{\mathbb{Z}^d}
\newcommand{\Ztd}{\mathbb{Z}^{2d}}
\newcommand{\Rt}{\R^2}
\newcommand{\Rtd}{\R^{2d}}
\newcommand{\Zp}{Z_p}
\newcommand{\al}{\alpha}
\newcommand{\be}{\beta}
\newcommand{\om}{\omega}
\newcommand{\Om}{\Omega}
\newcommand{\ga}{\gamma}
\newcommand{\Ga}{\Gamma}
\newcommand{\bga}{\boldsymbol\ga}
\newcommand{\bg}{\mathbf{g}}
\newcommand{\de}{\delta}
\newcommand{\la}{\lambda}
\newcommand{\La}{\Lambda}
\newcommand{\ala}{\la^\circ}
\newcommand{\aLa}{\La^\circ}
\newcommand{\nat}{\natural}
\newcommand{\G}{\mathcal{G}}
\newcommand{\bG}{\mathbf{G}}
\newcommand{\A}{\mathcal{A}}
\newcommand{\V}{\mathcal{V}}
\newcommand{\U}{\mathcal{U}}
\newcommand{\M}{\mathcal{M}}
\newcommand{\MV}{\mathcal{MV}}
\newcommand{\MP}{\mathcal{MP}}
\newcommand{\Sp}{\mathcal{S}}
\newcommand{\W}{\mathcal{W}}
\newcommand{\Hp}{\mathcal{H}}
\newcommand{\Mep}{M_\epsilon}
\newcommand{\Kep}{K_\epsilon}
\newcommand{\ka}{\kappa}
\newcommand{\cast}{\circledast}
\newcommand{\id}{\mbox{Id}}
\newcommand{\by}{\mathbf{Y}}
\newcommand{\bx}{\mathbf{X}}
\newcommand{\bz}{\mathbf{Z}}
\newcommand{\bn}{\mathbf{N}}
\newcommand{\bv}{\mathbf{v}}
\newcommand{\bu}{\mathbf{u}}
\newcommand{\bA}{\mathbf{A}}
\newcommand{\bC}{\mathbf{C}}
\newcommand{\bD}{\mathbf{D}}
\newcommand{\bh}{\mathbf{h}}
\newcommand{\bff}{\mathbf{f}}
\newcommand{\bH}{\mathbf{H}}
\newcommand{\bV}{\mathbf{V}}
\newcommand{\bX}{\mathbf{x}}
\newcommand{\byy}{\mathbf{y}}
\newcommand{\bZ}{\mathbf{z}}
\newcommand{\bt}{\mathbf{t}}
\newcommand{\bs}{\mathbf{\sigma}}
\newcommand{\bI}{\mathbf{I}}
\newcommand{\wpe}{w_{\psi_\epsilon}}

\newcommand{\lo}{{\ell^1}}
\newcommand{\Lo}{{L^1}}
\newcommand{\Lt}{{L^2}}
\newcommand{\SO}{{M^1}}
\newcommand{\SOc}{{S_{0,c}}}
\newcommand{\Lp}{{L^p}}
\newcommand{\Los}{{L^1_s}}
\newcommand{\WCl}{{W(C_0,\ell^1)}}
\newcommand{\lt}{{\ell^2}}

\newcommand{\conv}[2]{{#1}\,\ast\,{#2}}
\newcommand{\twc}[2]{{#1}\,\nat\,{#2}}
\newcommand{\mconv}[2]{{#1}\,\cast\,{#2}}
\newcommand{\set}[2]{\Big\{ \, #1 \, \Big| \, #2 \, \Big\}}
\newcommand{\inner}[2]{\langle #1,#2\rangle}
\newcommand{\innerBig}[2]{\Big \langle #1,#2 \Big \rangle}
\newcommand{\dotp}[2]{ #1 \, \cdot \, #2}

\newcommand{\Zpd}{\Zp^d}
\newcommand{\I}{\mathcal{I}}
\newcommand{\J}{\mathcal{J}}
\newcommand{\Zq}{Z_q}
\newcommand{\Zqd}{Zq^d}
\newcommand{\Zak}{\mathcal{Z}_{a}}
\newcommand{\C}{\mathbb{C}}
\newcommand{\F}{\mathcal{F}}

\newcommand{\convL}[2]{{#1}\,\ast_L\,{#2}}
\newcommand{\abs}[1]{\lvert#1\rvert}
\newcommand{\absbig}[1]{\big\lvert#1\big\rvert}
\newcommand{\absBig}[1]{\Big\lvert#1\Big\rvert}
\newcommand{\scp}[1]{\langle#1\rangle}
\newcommand{\norm}[1]{\lVert#1\rVert}
\newcommand{\normmix}[1]{\lVert#1\rVert_{\ell^{2,1}}}
\newcommand{\normbig}[1]{\big\lVert#1\big\rVert}
\newcommand{\normBig}[1]{\Big\lVert#1\Big\rVert}

\newcommand{\Conv}{\mathop{\scalebox{3}{\raisebox{-0.2ex}{$\ast$}}}}
\newcommand{\ConvN}{\mathop{\scalebox{3.5}{\raisebox{-0.2ex}{$\ast$}}}}
\newcommand{\Convtext}{\mathop{\scalebox{2.5}{\raisebox{-0.2ex}{$\ast$}}}}

\oddsidemargin  -0.31in
 \evensidemargin -0.31in
 \topmargin 0truept
 \headheight 5truept
 \headsep 0truept
 \footskip 20truept
 \textheight 235truemm
 \textwidth 180truemm
 \columnsep 8truemm

%%%%%%%%%%%%%%%%%%%%%%%%%%%%%%%%%%%%%%%%%%%%%%%%%%%%%%%%%%%%

\title{Gabor frames for model sets.}

\author{Ewa Matusiak}

\author{\small EWA MATUSIAK\footnote{Department of Mathematics, University of Vienna, Austria, ewa.matusiak@univie.ac.at}}

\maketitle
%%%%%%%%%%%%%%%%%%%%%%%%%%%%%%%%%%%%%%%%%%%%%%%%%%%%%%%%%%%

\begin{abstract}
We generalize three main concepts of Gabor analysis for lattices to the setting of model sets: Fundamental Identity of Gabor Analysis, Janssen's representation of the frame operator and Wexler-Raz biorthogonality relations. Utilizing the connection between model sets and almost periodic functions, as well as Poisson's summations formula for model sets we develop a form of a bracket product that plays a central role in our approach. Furthermore, we show that, if a Gabor system for a model set admits a dual which is of Gabor type, then the density of the model set has to be greater than one.
%\keywords{Gabor frames \and model sets \and almost periodic functions \and Poisson's summation formula}
\end{abstract}

\section{Introduction}\label{sec:intro}

One of the central themes within Gabor analysis for lattices is a duality theory for Gabor frames, including Wexler-Raz biorthogonality relations \cite{WR90} and Janssen's representation of a Gabor frame operator \cite{J95}. These results are closely connected with the so-called Fundamental Identity of Gabor Analysis, that can be derived from an application of Poisson's summation formula for the symplectic Fourier transform \cite{FL06}. These duality conditions allow us for example to specify whether a given Gabor system is a tight frame or whether two Gabor systems are dual to each other. An immediate consequence of the duality theory one can obtain necessary density conditions on a lattice so that a given Gabor system forms a frame. In this exposition we leave the setting of a lattice and consider a certain type of irregular sets of time-frequency shifts, namely model sets.

The first examples of model sets were studied by Meyer in \cite{Me72}. Meyer thought of model sets as generalizations of lattices which retain enough lattice-like structure to be useful for studying sampling problems in harmonic analysis \cite{MaMe10,Me12}. The crucial property of model sets is that there exists a form of Poisson's summation formula, which in turn will allow us to derive analogous duality theory as in the case of Gabor analysis for lattices.

First work on Gabor frames for model sets was done by Kreisel \cite{K16}, where he showed how Gabor frames for a simple model set can be made compatible with its topological dynamics and derived existence conditions for multiwindow Gabor frames for model sets. We will use some of his results here. A constructive approach, that is a characterization of tight and dual frames of semi-regular Gabor systems (where time shifts come from a lattice, and frequency shifts come from a model set, or vice versa), was recently obtained in \cite{M18}.

For general irregular sets of time-frequency shifts, that are however discrete and relatively separated, it is difficult to provide any constructive results as the tools to deal with such sets are missing. However, certain extensions into irregular Gabor frames were undertaken, for example in \cite{Ga09} or \cite{GOR15}. In \cite{Ga09}, the author gave a characterization of the weighted irregular Gabor tight frame and dual systems in $\Lt(\R^d)$ in terms of the distributional symplectic Fourier transform of a positive Borel measure on $\Rtd$ in the case where the window belongs to the Schwartz class. More recently in \cite{grrost17} the authors study nonuniform sampling in shift invariant spaces and construct semi-regular Gabor frames with respect to the class of totally positive functions. Their results are Beurling type results, expressed by means of density of the sampling sets.  

We utilize the connection between model sets and almost periodic functions and use harmonic analysis of the latter to develop a certain form of duality theory for Gabor frames for model sets. We rely strongly on Poisson's summation formula for model sets to introduce the so-called bracket product, in analogy to the bracket product for lattices introduced in \cite{BDR94} to study shift-invariant spaces, or later in multi-dimensional setting to study frames of translates \cite{L02, HLW02}. 

Almost periodic functions were recently investigated in the connection with Gabor frames in \cite{PU01,G04,BFG17}. As the space of almost periodic functions is non-separable, it can not admit countable frames, and the problem arises in which sense frame-type inequalities are still possible for norm estimation in this space \cite{G04,PU01,BFG17}. In \cite{BFG17} the authors also provide Gabor frames for a suitable separable subspaces of the space of almost periodic functions. We, on the other hand, use almost periodic functions as a tool to develop existence results for irregular Gabor frames for the space of square integrable functions.

The article is organized as follows. In Section~\ref{sec:not} we establish some notations and definitions that we will use throughout the article. In Section~\ref{sec:gabor frames lattice} we derive main identities of Gabor analysis for lattices using a different approach than the one presented in the literature, namely by constructing a certain bracket product. 
We introduce model sets in Section~\ref{sec:ms}, and we also shortly present main facts from the theory of almost periodic functions and point out some connections between the two. Section~\ref{sec:bracket}  is devoted to developing a technical tool, that is a bracket product for model sets, that we later use in Section~\ref{sec:main} to obtain Fundamental Identity of Gabor Analysis, Janssen representation and Wexler-Raz biorthogonality relations for Gabor systems for model sets. 

%%%%%%%%%%%%%%%%%%%%%%%%%%%%%%%%%%%%%%%%%%%%%%%%%%%%%%%%%%%%

\section{Notation and Preliminaries}\label{sec:not}

We will work with the Hilbert space of square integrable functions on $\Lt(\R^d)$. The key element in time-frequency analysis is the {\it time-frequency shift operator} $\pi(z)$, $z=(x,\om)\in\Rtd$, which acts on $\Lt(\R^d)$ by
\begin{equation*}
\pi(z)f(t) = M_{\om} T_x f(t) = e^{2\pi i \om \cdot t } f(t-x)\,.
\end{equation*} 
Here $M_\om$ denotes the modulation operator and $T_x$ the translation operator which are defined as 
\begin{equation*}
T_x f(t) = f(t-x) \quad \mbox{and} \quad M_\om f(f) = e^{2\pi i \om \cdot t} f(t)\,.
\end{equation*}
We define a Fourier transform on $\Lt(\R^d)$ as
\begin{equation*}
\widehat{f}(\zeta) = \int_{\R^d} f(t) e^{-2\pi i t \cdot \zeta}\, dt\,,
\end{equation*}
and denote the inverse Fourier transform of $f$ by $\check{f}$. In the sequel we will distinguish between Fourier transform on $\Lt(\R^d)$ and Fourier transform on $\Lt(\Rtd)$ by writing $\F f$ for the latter, $f\in\Lt(\Rtd)$.The Fourier transform has a property of interchanging translation and modulation, that is
\begin{equation*}
\widehat{T_x f} = M_{-x} \widehat{f}  \quad \mbox{and} \quad \widehat{M_\om f} = T_\om \widehat{f}\,.
\end{equation*}
The translation and modulation operators obey the following commutation relation
\begin{equation*}
M_\om T_x = e^{2\pi i x\cdot \om} T_x M_\om\,.
\end{equation*}
Combining the last two properties, we have that 
\begin{equation*}
\widehat{\pi(x,\om) f} = e^{2\pi i x\cdot \om} \pi(\om,-x) \widehat{f}\quad \mbox{and} \quad  \check{\pi(x,\om) f} = e^{2\pi i x\cdot \om} \pi(-\om,x) \check{f}\,.
\end{equation*}

Given a non-zero function $g\in\Lt(\R^d)$, the {\it short-time Fourier transform} of $f\in\Lt(\R^d)$ with respect to the window $g$, is defined as
\begin{equation*}
\V_g f (x,\om) := \int_{\R^d} f(t) \overline{g(t-x)} e^{-2\pi i \om \cdot t}\, dt = \inner{f}{\pi(z)g}\,, \quad z=(x,\om)\in\Rtd\,. 
\end{equation*}
We define the {\it modulation spaces} as follows: fix a non-zero Schwartz function $g\in \mathcal{S}(\R^d)$, and let
\begin{equation*}
M^p(\R^d) := \{ f\in \mathcal{S}'(\R^d): \V_g f \in L^p(\R^d)\}\,, \quad 1\leq p \leq \infty
\end{equation*}
with the norm $\norm{f}_{M^p} = \norm{\V_g f}_p$. Different choices for $g$ give rise to equivalent norms on $M^p(\R^d)$. For $p=2$, we have $M^2(\R^d) = \Lt(\R^d)$. The space $M^1(\R^d)$, known as Feichtinger's algebra, can also be characterized as
\begin{equation*}
M^1(\R^d) := \{ f\in\Lt(\R^d): \norm{\V_f f}_1 < \infty \}\,.
\end{equation*}

\begin{proposition}\cite{gr01}
The space $M^1(\R^d)$ has the following properties:
\begin{itemize}
\item[i)] $M^1(\R^d)$ is a Banach algebra under pointwise multiplication.
\item[ii)] $M^1(\R^d)$ is a Banach algebra under convolution.
\item[iii)] $M^1(\R^d)$ is invariant under time-frequency shifts. 
\item[vi)] $M^1(\R^d)$ is invariant under Fourier transform.
\end{itemize}
\end{proposition}
$M^1(\R^d)$ contains the Schwartz space $\mathcal{S}(\R^d)$ and it is dense in $L^p(\R^d)$, $1\leq p < \infty$, therefore it is a very useful space in the time-frequency analysis. This will be the space of windows $g$. 

Another collection of function spaces that will be useful in our calculations are the amalgam spaces $\mathbf{W}(L^p,\ell^q)(\R^d)$, $1\leq p,q \leq \infty$. The amalgam space $\mathbf{W}(L^p,\ell^q)(\R^d)$ is the space of functions $f$ such that
\begin{equation*}
\norm{f}_{\mathbf{W}(L^p,\ell^q)(\R^d)} := \left (\sum_{k\in\Z^d} \norm{f\mathds{1}_k}_p^q \right)^{1/q}\,, 
\end{equation*}
where $\mathds{1}_k$ is the characteristic function of the cube $[0,1]^d +k$, $k\in\Z^d$. Different partitions of $\R^d$ give equivalent norms. The space $\mathbf{W}(L^{\infty},\lo)(\R^d)$, which is a subspace of $L^1(\R^d)$, is referred to as Wiener's algebra, and will be denoted by $\mathbf{W}(\R^d)$. It consists of all $f\in L^\infty(\R^d)$ such that
\begin{equation*}
\norm{f}_{\mathbf{W}(\R^d)} = \norm{f}_{\mathbf{W}(L^\infty,\lo)(\R^d)} := \sum_{k\in\Z^d} \norm{f\mathds{1}_k}_\infty\,.
\end{equation*}
Moreover, we have $\mathbf{W}(\R^d) \subseteq \mathbf{W}(L^{\infty},\lt)(\R^d)$. The following convolution relation will be often used:
\begin{equation*}
L^1 \ast \mathbf{W} \subseteq \mathbf{W}\,.
\end{equation*}
For a function $f\in M^1(\R^d)$ it follows that $f\in\mathbf{W}(L^{\infty},\lo)(\R^d)$ and $\widehat{f} \in \mathbf{W}(L^{\infty},\lo)(\R^d)$, \cite{gr01}.

Let $C_0(\R^d)$ be the space of all continuous functions that vanish at infinity. Then the closed subspace of $\mathbf{W}(\R^d)$ consisting of continuous functions is $\mathbf{W}(C_0,\lo)(\R^d)$. Continuity of elements in $\mathbf{W}(L^{\infty},\lo)(\R^d)$ allows for pointwise evaluations, and we have (see Proposition~$11.1.4$ in \cite{gr01}): if $F\in \mathbf{W}(C_0,\lo)(\R^d)$, then $F|\La \in \lo(\La)$, for $\La$ any discrete relatively separated set in $\R^d$ with the norm estimate
\begin{equation*}
\sum_{\la\in\La} \abs{F(\la)} \leq \text{rel}(\La) \norm{F}_{\mathbf{W}(\R^d)}\,.
\end{equation*}
A discrete subset $\La$ of $\R^d$ is relatively separated if
\begin{equation*}
\text{rel}(\La) := \text{sup} \{ \# \{\La \cap B(x,1)\}:x\in\R^d\} < \infty\,,
\end{equation*}
where $B(x,1)$ is a ball of radius $1$ in $\R^d$ centered at $x$.

There are two more time frequency representations of $f$ that we will be using. For $f,g\in \Lt(\R^d)$, the {\it cross-ambiguity function} of $f$ and $g$
\begin{equation*}
\A(f,g) (x,\om) = \int_{\R^d} f\Big(t + \frac{x}{2}\Big) \overline{g\Big (t-\frac{x}{2}\Big)} e^{-2 \pi i t \cdot \om}\, dt\,,
\end{equation*}
and the {\it cross-Wigner distribution} of $f$ and $g$
\begin{equation*}
\W (f,g) (x,\om) = \int_{\R^d} f\Big(x + \frac{t}{2}\Big) \overline{g\Big(x-\frac{t}{2}\Big)} e^{-2 \pi i t \cdot \om}\, dt\,.
\end{equation*}
The three time-frequency representations, $\V_gf$, $\A(f,g)$ and $\W(f,g)$ are related to each other. Before we state the relationships, we define the rotation $\U = \U_J$ of a function $F$ on $\Rtd$ as
\begin{equation*}
\U F (x,\om) = \U_J F (x,\om) = F(J(x,\om)^T) = F(\om,-x) \quad \mbox{with} \quad J=\left ( \begin{array}{cc} 0 & I \\ -I & 0 \end{array}\right )\,,
\end{equation*}
where $I$ is the $d\times d$ identity matrix. Then $\U^{-1} = \U_J^{-1} = \U_{-J}$. In the following two propositions we list the properties of the cross-ambiguity function and the cross-Wigner distribution that we will be using throughout the exposition. For the proofs we refer the reader to \cite{gr01}.
\begin{proposition}\label{prop:ambiguity}
For $f,g\in\Lt(\R^d)$ the cross-ambiguity function has the following properties.
\begin{itemize}
\item[a)] $\A(f,g)$ is uniformly continuous on $\Rtd$.
\item[b)] $\A(f,g)(x,\om) = e^{\pi i x \cdot \om} \V_g f(x,\om)$. 
\item[c)] $\A(f,g)^* = \A(g,f)$, where $\A(f,g)^*(x,\om) = \overline{\A(f,g)(-x,-\om)}$ is the involution.  
\item[d)] $\A(\pi(x,\om)f,g)(t,\zeta) = e^{\pi i t \cdot \om} e^{-\pi i x\cdot (\zeta-\om)} \A(f,g) (t-x,\zeta-\om)$ for $(x,\om)\in\Rtd$.
\item[e)] $\A(f,g) = \F^{-1} \U^{-1} \W(f,g)$.
\end{itemize}
Moreover, if $f,g\in M^1(\R^d)$, then $\A(f,g) \in \mathbf{W}(C_0,\lo)(\Rtd)$, and by $b)$, also $\V_g f \in \mathbf{W}(C_0,\lo)(\Rtd)$.
\end{proposition}

\begin{proposition}\label{prop:wigner}
For $f,g\in\Lt(\R^d)$ the cross-Wigner distribution has the following properties.
\begin{itemize}
\item[a)] $\W(f,g)$ is uniformly continuous on $\Rtd$.
\item[b)] $\W(f,g) = \overline{\W(g,f)}$.
\item[c)] $\W(\widehat{f},\widehat{g}) = \U^{-1} W(f,g)$.
\item[d)]  $\W(f,g) = \F \U \A(f,g)$.
\item[e)] For $(x,\om)\in \Rtd$ and $(x',\om')\in\Rtd$, we have
\begin{align*}
\W(T_x M_\om f, &T_{x'} M_{\om'}g)(t,\zeta) \\
&= e^{-\pi i (x+x')\cdot (\om-\om')} e^{2\pi i t\cdot (\om-\om')} e^{-2\pi i \zeta\cdot (x-x')} \W(f,g)\Big (t - \frac{x+x'}{2},\zeta - \frac{\om+\om'}{2}\Big)
\end{align*}
\item[f)] Moyal's formula: for $f_1,g_1 \in\Lt(\R^d)$,
\begin{equation*}
\inner{W(f,g)}{W(f_1,g_1)}_{\Lt(\Rtd)} = \inner{f}{f_1} \overline{\inner{g}{g_1}}\,.
\end{equation*}
\end{itemize}
Moreover, if $f,g\in M^1(\R^d)$, then $\W(f,g) \in \mathbf{W}(C_0,\lo)(\Rtd)$.
\end{proposition}
We note here that Moyal's formula also holds for the cross-ambiguity function and the short-time Fourier transform. We also mention the following tensor-product properties: if $Z=(z,\tilde{z})$ with $z\in \R^{2n}$, $\tilde{z}\in\R^{2d-2n}$ and $\psi_1,\phi_1\in M^1(\R^n)$, $\psi_2,\phi_2\in M^1(\R^{d-n})$, then
\begin{equation*}
\A(\psi_1\otimes \psi_2,\phi_1\otimes\phi_2)(Z) = \big (\A(\psi_1,\phi_1) \otimes \A(\psi_2,\phi_2)\big)(z,\tilde{z}) = \A(\psi_1,\phi_1)(z)   \A(\psi_2,\phi_2)(\tilde{z})\,.
\end{equation*}
Similarly for the cross-Wigner distribution and short-time Fourier transform.

Before we turn to Gabor systems, we need one more result. It is originally stated for the short-time Fourier transform, but we state it here for the cross-ambiguity function. The proof is the same.
\begin{lemma}\cite{Fei81}\label{lem:AinS}
Let $f,g\in M^1(\R^d)$. Then $\A(f,g) \in M^1(\Rtd)$.
\end{lemma}

The main object of our study here are Gabor systems. Let $\La$ be a relatively separated subset of $\Rtd$. A {\it Gabor system} is a collection of time-frequency shifts of one or more window functions $g_i\in M^1(\R^d)$, $i=1,\ldots,M$, with respect to $\La \subset \Rtd$, and it is denoted by
\begin{equation*}
\G(g_1,\ldots,g_M;\La) = \{ \pi(\la)g_i:  i=1,\ldots,M,\, \la\in\La\}\,.
\end{equation*}
The analysis operator and the synthesis operator for a Gabor system $\G(g_1,\ldots,g_M;\La)$ are defined as
\begin{align*}
C_{g,\La}f &:= (\inner{f}{\pi(\la)g_i})_{\la\in\La; i=1,\ldots,M}\\
D_{g,\La} c &:= \sum_{i=1}^M \sum_{\la\in\La} c_\la \pi(\la)g_i
\end{align*}
are bounded between $M^p(\R^d)$ and $\ell^p(\La)$ spaces, with estimates
\begin{align*}
\norm{C_{g,\La}f}_{\ell^p} &\leq  \text{rel}(\La) \,\norm{f}_{M^p} \, \max_{i=1,\ldots,M} \norm{g_i}_{M^1} \\
\norm{D_{g,\La} c}_{M^p} &\leq   \text{rel}(\La)\, \norm{c}_{\ell^p}\, \max_{i=1,\ldots,M} \norm{g_i}_{M^1} 
\end{align*}
A Gabor system $\G(g_1,\ldots,g_M;\La)$ with $g_i\in M^1(\R^d)$, $i=1,\ldots,M$, will be called an {\it $M^p$-frame} if $C_{g,\La}$ is bounded below on $M^p(\R^d)$. This is equivalent to having constants $A_g,B_g>0$ so that for all $f\in M^p(\R^d)$
\begin{equation*}
\sqrt{A_g} \norm{f}_{M^p} \leq \norm{S_g^\La f}_{M^p} \leq \sqrt{B_g} \norm{f}_{M^p}\,,
\end{equation*}
where $S_g^{\La}$ is a {\it frame operator} given by
\begin{equation}\label{eq:frame operator}
S_g^\La f = \sum_{i=1}^M \sum_{\la\in\La} \inner{f}{\pi(\la)g_i} \pi(\la) g_i\,.
\end{equation}
In this case the frame operator is invertible on $M^p(\R^d)$. Theorem~$3.2$ in \cite{GOR15} states that when each $g_i\in M^1(\R^d)$, $i=1,\ldots,M$, then $\G(g_1,\ldots,g_M;\La)$ is an $M^p$-frame for some $p\in [1,\infty]$, if and only if it is an $M^p$-frame for all $p$. The constants $A_g$ and $B_g$ are called lower and upper frame bounds, respectively. If $A_g=B_g$ then the frame is called a {\it tight Gabor frame}, and if $A_g=B_g=1$, a {\it normalized tight Gabor frame}. 

When $\G(g_1,\ldots,g_M;\La)$ is an $M^p$-frame for some $p\in [1,\infty]$, then we have a frame decomposition
\begin{equation}\label{eq:frame decomposition}
f = \sum_{i=1}^M \sum_{\la\in\La}  \inner{f}{\big (S_g^{\La} \big)^{-1} \pi(\la) g_i} \pi(\la) g_i\,, \quad \mbox{for all $f\in M^p(\R^d)$.}
\end{equation}
The sequence $\{\big(S_g^\La \big )^{-1} \pi(\la) g_i: i=1,\ldots, M,\, \la\in\La\}$ is also a frame for $M^p(\R^d)$, called the {\it canonical dual frame} of $\G(g_1,\ldots,g_M;\La)$, and has upper and lower frame bounds $B_g^{-1}$ and $A_g^{-1}$, respectively. If the frame is tight, then $\big (S_g^\La \big )^{-1}= A_g^{-1} I$, where $I$ is the identity operator, and the frame decomposition becomes
\begin{equation*}
f = A_g^{-1}\,\sum_{i=1}^M \sum_{\la\in\La} \inner{f}{\pi(\la) g_i} \pi(\la) g_i\,,\quad \mbox{for all $f\in M^p(\R^d)$}\,.
\end{equation*}
In order to use the representation \eqref{eq:frame decomposition} in practice, we need to be able to calculate $\big (S_g^\La \big )^{-1}$. While the existence of $\big (S_g^\La \big )^{-1}$ is guaranteed by the frame condition, it is usually tedious to find this operator explicitly. Moreover, if $\La$ is not a lattice in $\R^d$, then the frame operator of $S_g^{\La}$ does not commute with time-frequency shifts, that is $\pi(\be) S_g^\La \neq S_g^\La \pi(\be)$ for $\be\in\La$. Indeed, let $\be\in\La$ and $f\in \Lt(\R^d)$, then
\begin{equation*}
S_g^{\La} \pi(\be) f = \sum_{i=1}^M \sum_{\la\in\La} \inner{\pi(\be)f}{\pi(\la) g_i} \pi(\la) g_i= \sum_{i=1}^M \sum_{\la\in\La} e^{-2\pi i \be_1\cdot (\la_2-\be_2)} \inner{f}{\pi(\la-\be) g_i} \pi(\la) g_i\,,
\end{equation*}
where $\be=(\be_1,\be_2)$ and $\la=(\la_1,\la_2)$. On the other hand,
\begin{align*}
\pi(\be) S_g^{\La} f &= \sum_{i=1}^M \sum_{\la\in\La} \inner{f}{\pi(\la) g_i} \pi(\be) \pi(\la) g_i= \sum_{i=1}^M \sum_{\la\in\La} e^{-2\pi i \be_1\cdot \la_2} \inner{f}{\pi(\la) g_i} \pi(\la+\be) g_i\\
&= \sum_{i=1}^M \sum_{\la\in\La+\be} e^{-2\pi i \be_1\cdot (\la_2-\be_2)} \inner{f}{\pi(\la-\be) g_i} \pi(\la) g_i \,,
\end{align*}
and the two expressions are not equal since $\La+\be \neq \La$. Therefore, the canonical dual frame $\{ \big (S_g^{\La}\big)^{-1} \pi(\la) g_i:i=1,\ldots,M,\, \la\in\La\}$ does not have the same structure as $\G(g_1,\ldots,g_M;\La)$, that is it is not a Gabor frame, and, in order to compute the canonical dual frame we would have to apply $\big (S_g^{\La} \big)^{-1}$ to $\pi(\la) g_i$, for all $i=1,\ldots,M$ and all $\la\in\La$. Hence, we search for a pair of dual frames, rather than just one frame. Let $\G(g_1,\ldots,g_M;\La)$ and $\G(h_1,\ldots,h_M;\La)$ be Gabor systems, then we can define a mixed frame operator
\begin{equation*}
S_{g,h}^{\La} f = \sum_{i=1}^M \sum_{\la\in\La} \inner{f}{\pi(\la) g_i} \pi(\la) h_i
\end{equation*}
which is a bounded linear operator on $ M^p(\R^d)$. If $S_{g,h}^{\La} f = f$ for every $f\in M^p(\R^d)$, then we call $\G(h_1,\ldots,h_M;\La)$ a {\it generalized dual Gabor frame} of $\G(g_1,\ldots,g_M;\La)$.

%%%%%%%%%%%%%%%%%%%%%%%%%%%%%%%%%%%%%%%%%%%%%%%%%%%%%%%%%%%%

\section{Gabor frames for lattices: revised}\label{sec:gabor frames lattice}

Before we turn our attention to Gabor frames for model sets, we revisit here known results for regular Gabor frames, that is where time-frequency shifts come from a lattice. We present a different approach then the one presented in the literature, by constructing a bracket product for the time-frequency plane. Alternative, but in some sense similar approaches, were also developed in \cite{FL06} and more recently in \cite{JaLe16}.

Throughout this section $\La$ will be a lattice, that is a discrete subgroup of $\Rtd$. A lattice can be represented by an invertible matrix $A\in GL(2d,\R)$ and is then given by $\La = A \Ztd$. We define the volume of a lattice $\La=A\Ztd$ by
$\text{vol}(\La) = \abs{\det (A)}$. The density of a lattice is given by the reciprocal of the volume, that is $D(\La) = \text{vol}(\La)^{-1}$. A dual lattice is defined as $\La^\ast=A^{-T}\Ztd$.

A tool that is heavily utilized in time-frequency analysis is the Poisson summation formula for functions on $\R^d$. However, we will use here Poisson summation formula for functions on $\Rtd$.

\begin{theorem}[Poisson Summation Formula for Lattices]\label{thm:poisson}
Let $\La$ be a lattice in $\Rtd$. Then, for every $F\in M^1(\Rtd)$ we have
\begin{equation*}
\sum_{\la\in\La} F(\la) e^{-2\pi i \la \cdot z} = \text{vol}(\La)^{-1} \sum_{\la^*\in \La^*} \F F(z-\la^*)\,, 
\end{equation*}
where $\La^\ast$ is a dual lattice. The identity holds pointwise for all $z\in\Rtd$, and both sums converge uniformly and absolutely  $z\in\Rtd$.
\end{theorem}

Let $f,g,h\in M^1(\R^d)$. Then, by Lemma~\ref{lem:AinS} $\A(f,g)$ and $\W(f,g)$ belong to $M^1(\Rtd)$, and by Poisson summation formula,
\begin{equation}\label{eq:A-W}
\sum_{\la\in\La} \A(f,g)(\la) e^{-2\pi i \la \cdot z} = \text{vol}(\La)^{-1} \sum_{\la^*\in \La^*} \W(\widehat{f},\widehat{g})(z-\la^*)\,,
\end{equation}
where we used the relation $\A(f,g) = \F^{-1} \W(\widehat{f},\widehat{g})$ derived from Proposition~\ref{prop:ambiguity} and Proposition~\ref{prop:wigner}. Assume that $\La= A\Ztd$, then we can write \eqref{eq:A-W} more explicitly as
\begin{equation}
\sum_{n\in\Ztd} \A(f,g)(An) e^{-2\pi i An \cdot z} = (\det A)^{-1} \sum_{n\in \Ztd} \W(\widehat{f},\widehat{g})(z-A^{-T}n)\,.
\end{equation} 

We are now in the position to define a main ingredient in our approach, the bracket product. Let $f\in \Lt(\R^d)$ and $g\in M^1(\R^d)$. For a fixed $z=(x,\om)\in\Rtd$, the {\it generalized $\La-$bracket product} of $\pi(z)f$ and $g$ is defined as
\begin{equation}\label{eq:bracket La}
\big [\widehat{\pi(z)f},\widehat{g}\big ]_\La(\tilde{z}) = \text{vol}(\La)^{-1} \sum_{\la^*\in\La^*} M_{-z} \W(\widehat{f},\widehat{g})(\tilde{z}-\la^*)\,,
\end{equation}
where $M_z$ denotes the $\Rtd$ modulation by $z$. It follows from Monotone Convergence Theorem and the fact that $\W(\widehat{f},\widehat{g}) \in \Lo(\Rtd)$, that the series  \eqref{eq:bracket La} converges absolutely to a function $\Lo(\T_{\La^*})$, $\T_{\La^*} = \Rtd/\La^*$. When $\La$ is represented by a matrix $A$, we have $\T_{\La^*} = A^{-1}\T$ where $\T$ is the torus in $\Rtd$. Since $\big [\widehat{\pi(z)f},\widehat{g}\big ]_\La \in \Lo(\T_{\La^*})$, we can compute the Fourier coefficients
\begin{align*}
\int_{\T_{\La^*}} \big [\widehat{\pi(z)f},\widehat{g}\big ]_\La(\tilde{z}) e^{2\pi i \tilde{z} \cdot \la}\,d\tilde{z} &=  \text{vol}(\La)^{-1} \int_{\T_{\La^*}} \sum_{\la^*\in\La^*} M_{-z} \W(\widehat{f},\widehat{g})(\tilde{z}-\la^*) e^{2\pi i \tilde{z} \cdot \la}\,d\tilde{z} \\
&= \int_{\Rtd} M_{-z} \W(\widehat{f},\widehat{g})(\tilde{z}) e^{2\pi i \tilde{z} \cdot \la}\,d\tilde{z} = \F^{-1} M_{-z} \W(\widehat{f},\widehat{g})(\la) \\
&= T_z \F^{-1} \W(\widehat{f},\widehat{g}) (\la) = T_z \A(f,g)(\la)\\
&= \A(f,g)(\la-z)\,.
\end{align*}
Then the application of the Plancherel theorem for Fourier series, gives us the following proposition.
\begin{proposition}\label{prop:Planch}
Let $\La$ be  a lattice in $\Rtd$. Fix $z\in \Rtd$. Then for all $g,h\in M^1(\Rtd)$ and $f_1,f_2\in \Lt(\Rtd)$, we have
\begin{equation*}
\sum_{\la \in\La} \A(f_1,g)(\la-z) \overline{\A(f_2,h)(\la-z)} = \int_{\T_{\La^*}} \big [\widehat{\pi(z)f_1},\widehat{g}\big ]_\La(\tilde{z}) \, \overline{\big [\widehat{\pi(z)f_2},\widehat{h}\big ]_\La(\tilde{z})} \, d\tilde{z}\,,
\end{equation*}
with the absolute convergence of the integral.
\end{proposition}

The following two results are the main ingredients in deriving Janssen representation of the Gabor frame operator $S_{g,h}^\La$ and successive characterization of tight and dual Gabor frames.
\begin{proposition}\label{prop:function F lattice}
Let $\La$ be  a lattice in $\Rtd$. Assume that $f_1,f_2,g,h\in M^1(\R^d)$. Then the function
\begin{equation}\label{eq:function F lattice}
F(z,\tilde{z}) = \big [\widehat{\pi(z)f_1},\widehat{g}\big ]_\La(\tilde{z}) \, \overline{\big [\widehat{\pi(z)f_2},\widehat{h}\big ]_\La(\tilde{z})}
\end{equation}
is continuous and periodic, and coincides pointwise with its Fourier series 
\begin{equation}\label{eq:Fourier F lattice}
 \text{vol}(\La)^{-1} \sum_{\be^* \in\La^*} \sum_{\la\in\La} \innerBig{\A(f_1,g)}{T_\la M_{\be^*} \A(f_2,h)}_{\Lt(\Rtd)}\, e^{-2\pi i \la \cdot \tilde{z}} \, e^{-2\pi i z \cdot \be^*}\,.
\end{equation}
\end{proposition}

\begin{proof}
Let $\La$ be  a lattice and $f_1,f_2,g,h\in M^1(\R^d)$. Using Proposition~\ref{prop:Planch} and a change of index, we write explicitely
\begin{align*}
F(z,\tilde{z}) &= \big [\widehat{\pi(z)f_1},\widehat{g}\big ]_\La(\tilde{z}) \, \overline{\big [\widehat{\pi(z)f_2},\widehat{h}\big ]_\La(\tilde{z})} \, d\tilde{z}\\
&= \text{vol}(\La)^{-2} \sum_{\la^*,\be^* \in\La^*} M_{-z} \W(\widehat{f_1},\widehat{g})(\tilde{z}-\la^*) \overline{M_{-z} \W(\widehat{f_2},\widehat{h})(\tilde{z}-\be^*)}\\
&= \text{vol}(\La)^{-2} \sum_{\la^*,\be^* \in\La^*}  \W(\widehat{f_1},\widehat{g})(\tilde{z}-\la^*) \overline{\W(\widehat{f_2},\widehat{h})(\tilde{z}-\be^*)}\, e^{-2\pi i z \cdot (\be^*- \la^*)}\\
&= \text{vol}(\La)^{-2} \sum_{\la^*,\be^* \in\La^*}  \W(\widehat{f_1},\widehat{g})(\tilde{z}-\la^*) \overline{\W(\widehat{f_2},\widehat{h})(\tilde{z}-\la^* - \be^*)} \, e^{-2\pi i z \cdot \be^*}\\
&= \text{vol}(\La)^{-1} \sum_{\be^* \in\La^*} \left [ \text{vol}(\La)^{-1} \sum_{\la^* \in\La^*}  \W(\widehat{f_1},\widehat{g})(\tilde{z}-\la^*) \overline{T_{\be^*}\W(\widehat{f_2},\widehat{h})(\tilde{z}-\la^*)} \right ] \, e^{-2\pi i z \cdot \be^*}\,.
\end{align*}
For a fixed $\be^*\in\La^\ast$, consider the series
\begin{equation*}
F_{\be^*}(\tilde{z}) =\text{vol}(\La)^{-1} \sum_{\la^* \in\La^*}  \W(\widehat{f_1},\widehat{g})(\tilde{z}-\la^*) \overline{T_{\be^*}\W(\widehat{f_2},\widehat{h})(\tilde{z}-\la^*)}\,.
\end{equation*}
Since $\W(\widehat{f_1}, \widehat{g})$ and $\W(\widehat{f_2}, \widehat{h})$ lie in $M^1(\Rtd)$, their product as well. Moreover, 
\begin{equation*}
\F^{-1}\big (\W(\widehat{f_1}, \widehat{g}) \overline{T_{\be^*}\W(\widehat{f_2}, \widehat{h})}\big ) = \A(f_1,g) \ast (M_{\be^*} \A(f_2,h))^*  \in M^1(\Rtd)\,,
\end{equation*}
since $\A(f_1,g),\A(f_2,h) \in M^1(\Rtd)$. Therefore, by Poisson summation formula we can write $F_{\be^*}(\tilde{z})$ as
\begin{align*}
F_{\be^*}(\tilde{z}) &= \sum_{\la \in \La} \big (\A(f_1,g) \ast  (M_{\be^*} \A(f_2,h))^*\big )(\la)\, e^{-2\pi i \la \cdot \tilde{z}}\\
&= \sum_{\la\in\La} \innerBig{\A(f_1,g)}{T_\la M_{\be^*} \A(f_2,h)}_{\Lt(\Rtd)}\, e^{-2\pi i \la \cdot \tilde{z}}\,,
\end{align*}
and $F(z,\tilde{z})$ becomes
\begin{equation}\label{eq:series F}
F(z,\tilde{z}) =  \text{vol}(\La)^{-1} \sum_{\be^* \in\La^*} \sum_{\la\in\La} \innerBig{\A(f_1,g)}{T_\la M_{\be^*} \A(f_2,h)}_{\Lt(\Rtd)}\, e^{-2\pi i \la \cdot \tilde{z}} \, e^{-2\pi i z \cdot \be^*}\,.
\end{equation}
By Lemma~\ref{lem:AinS}, $\A(f_1,g),\A(f_2,h) \in M^1(\Rtd)$, and therefore, 
\begin{equation*}
\V_{\A(f_2,h)}( \A(f_1,g))(w,\tilde{w})= e^{2\pi i w \cdot \tilde{w} }\innerBig{\A(f_1,g)}{\pi(w,\tilde{w}) \A(f_2,h)} \in \mathbf{W}(C_0,\ell^1)(\Rtd\times \Rtd)\,.
\end{equation*} 
Hence, $\V_{\A(f_2,h)}( \A(f_1,g))$ restricted to $\La \times \La^*$ belongs to $\lo(\La\times \La^*)$, and as a consequence, the series \eqref{eq:series F} defining $F(z,\tilde{z})$ is absolutely convergent. By the uniqueness of the Fourier series,  \eqref{eq:series F} is the Fourier series of $F$.
\end{proof}

\begin{proposition}\label{prop:function N lattice}
Let $\La$ be  a lattice in $\Rtd$. Assume that $g_i,h_i\in M^1(\Rtd)$, for every $i=1,\ldots,M$. Then for every $f_1,f_2\in M^1(\Rtd)$, the function
\begin{equation}\label{eq:function N lattice}
\mathcal{N}(z) = \sum_{i=1}^M \sum_{\la\in\La} \inner{\pi(z)f_1}{\pi(\la)g_i}\overline{\inner{\pi(z)f_2}{\pi(\la)h_i}}
\end{equation}
is continuous and periodic, and coincides pointwise with its Fourier series $\sum_{\la^\circ \in\La^\circ} \widehat{\mathcal{N}}(\la^\circ) e^{-2\pi i \sigma(\la^\circ, z)}$, with
\begin{equation}\label{eq:coeff N lattice}
\widehat{\mathcal{N}}(\la^\circ) = \text{vol}(\La)^{-1} \sum_{i=1}^M \inner{h_i}{\pi(\la^\circ) g_i} \inner{\pi(\la^\circ) f_1}{f_2}\,,
\end{equation}
where $\La^\circ = J\La^*$ is the adjoint lattice and $\sigma$ a symplectic form defined as $\sigma(\theta,z) = \theta \cdot Jz$, with $\theta$ and $z$ seen as column vectors in $\Rtd$.
\end{proposition}

\begin{proof}
Let $\La$ be a lattice in $\Rtd$ and choose $f_1,f_2 \in M^1(\Rtd)$. Then the function $\mathcal{N}(z)$ is well defined and can be expressed using the mixed frame operator $S_{g,h}^\La$ as $\mathcal{N}(z) = \inner{S_{g,h}^{\La} \pi(z)f_1}{\pi(z)f_2}$. Moreover, using the relations between short time Fourier transform and cross-ambiguity function, we can write $\mathcal{N}$ as
\begin{align*}
\mathcal{N}(z) &=  \sum_{i=1}^M \sum_{\la\in\La} \V_{g_i} (\pi(z)f_1 )(\la) \overline{\V_{h_i}(\pi(z)f_2)(\la)} = \sum_{i=1}^M \sum_{\la\in\La} \A(\pi(z)f_1,g_i)(\la) \overline{\A(\pi(z)f_2,h_i)(\la)} \\
&= \sum_{i=1}^M \sum_{\la\in\La} \A(f_1,g_i)(\la-z) \overline{\A(f_2,h_i)(\la-z)}\,.
\end{align*}
For $i=1,\ldots,M$ fixed, let
\begin{equation*}
\mathcal{N}_i(z) = \sum_{\la\in\La} \A(f_1,g_i)(\la-z) \overline{\A(f_2,h_i)(\la-z)}\,,
\end{equation*}
which is a periodic function. Using Proposition~\ref{prop:Planch} and Proposition~\ref{prop:function F lattice}, we have
\begin{align*}
\mathcal{N}_i(z) &= \int_{\T_{\La^*}} \big [\widehat{\pi(z)f_1},\widehat{g_i}\big ]_\La(\tilde{z}) \, \overline{\big [\widehat{\pi(z)f_2},\widehat{h_i}\big ]_\La(\tilde{z})} \, d\tilde{z}\\
&= \text{vol}(\La)^{-1}  \int_{\T_{\La^*}} \sum_{\be^* \in\La^*} \sum_{\la\in\La} \innerBig{\A(f_1,g)}{T_\la M_{\be^*} \A(f_2,h)}_{\Lt(\Rtd)}\, e^{-2\pi i \la \cdot \tilde{z}} \, e^{-2\pi i z \cdot \be^*}\, d\tilde{z}\\
&=  \text{vol}(\La)^{-1} \sum_{\be^* \in\La^*}  \left [ \int_{\T_{\La^*}} \sum_{\la\in\La} \innerBig{\A(f_1,g)}{T_\la M_{\be^*} \A(f_2,h)}_{\Lt(\Rtd)}\, e^{-2\pi i \la \cdot \tilde{z}} \, d\tilde{z} \right ] \, e^{-2\pi i z \cdot \be^*}\\
&= \text{vol}(\La)^{-1} \sum_{\be^* \in\La^*} \innerBig{\A(f_1,g)}{M_{\be^*} \A(f_2,h)}_{\Lt(\Rtd)} \, e^{-2\pi i z \cdot \be^*}\\
&= \text{vol}(\La)^{-1} \sum_{\be^* \in\La^*} \innerBig{\W(\widehat{f_1},\widehat{g_i})}{T_{\be^*}\W(\widehat{f_2},\widehat{h_i})}_{\Lt(\Rtd)} \, e^{-2\pi i z \cdot \be^*}\,.
\end{align*}
The interchange of the integral and the sum is possible due to the Fubini's Theorem. Now, since $\innerBig{\A(f_1,g)}{T_\la M_{\be^*} \A(f_2,h)}_{\Lt(\Rtd)}$ are in $\lo(\La\times\La^*)$ by the proof of Proposition~\ref{prop:function F lattice}, the coefficients $\innerBig{\W(\widehat{f_1},\widehat{g_i})}{T_{\be^*}\W(\widehat{f_2},\widehat{h_i})}_{\Lt(\Rtd)}$ are in $\ell^1(\La^*)$. Hence, for each $i=1,\ldots,M$, the function $\mathcal{N}_i$ is continuous, as it equals the absolutely convergent trigonometric series
\begin{equation*}
\mathcal{N}_i(z) =  \text{vol}(\La)^{-1} \sum_{\be^* \in\La^*}  \innerBig{\W(\widehat{f_1},\widehat{g_i})}{T_{\be^*}\W(\widehat{f_2},\widehat{h_i})}_{\Lt(\Rtd)}\,e^{-2\pi i z \cdot \be^*}\,.
\end{equation*}
The coefficients $\innerBig{\W(\widehat{f_1},\widehat{g_i})}{T_{\be^*}\W(\widehat{f_2},\widehat{h_i})}_{\Lt(\Rtd)}$ can be simplified. Let $\frac{\be^*}{2} = (\frac{\be^*_1}{2},\frac{\be^*_2}{2})$, then using Proposition~\ref{prop:wigner}~$e)$, we have
\begin{align*}
 T_{\frac{-\be^*}{2}}\W(\widehat{f_1},\widehat{g_i})(t,\zeta) &= e^{-\pi i \be_1^* \cdot \be_2^*}  e^{-2\pi i t\cdot\be_2^*} e^{2\pi i \zeta \cdot \be_1^*} \W(\widehat{f_1}, T_{-\be_1^*} M_{-\be_2^*} \widehat{g_i})(t,\zeta)\,,\\
 T_{\frac{\be^*}{2}}\W(\widehat{f_2},\widehat{h_i})(t,\zeta) &= e^{\pi i \be_1^* \cdot \be_2^*}  e^{-2\pi i t\cdot\be_2^*} e^{2\pi i \zeta \cdot \be_1^*} \W( T_{\be_1^*} M_{\be_2^*}\widehat{f_2},\widehat{h_i})(t,\zeta)\,,
\end{align*}
and applying Moyal's formula we obtain
\begin{align*}
 \innerBig{\W(\widehat{f_1},\widehat{g_i})}{T_{\be^*}\W(\widehat{f_2},\widehat{h_i})}_{\Lt(\Rtd)} &=  \innerBig{T_{\frac{-\be^*}{2}} \W(\widehat{f_1},\widehat{g_i})}{T_{\frac{\be^*}{2}}\W(\widehat{f_2},\widehat{h_i})}_{\Lt(\Rtd)} \\
 &= e^{-2\pi i \be_1^* \cdot \be_2^*} \innerBig{\W(\widehat{f_1}, T_{-\be_1^*} M_{-\be_2^*} \widehat{g_i})}{\W(T_{\be_1^*} M_{\be_2^*} \widehat{f_2},\widehat{h_i})}_{\Lt(\Rtd)}\\
 &= e^{-2\pi i \be_1^* \cdot \be_2^*} \inner{\widehat{f_1}}{T_{\be_1^*} M_{\be_2^*} \widehat{f_2}} \overline{\inner{T_{-\be_1^*} M_{-\be_2^*} \widehat{g_i}}{\widehat{h_i}}}\\
%&= e^{-2\pi i \be_1^* \cdot \be_2^*} \inner{\widehat{f}}{T_{\be_1^*} M_{\be_2^*} \widehat{f}} \inner{\widehat{h_i}}{T_{-\be_1^*} M_{-\be_2^*} \widehat{g_i}}\\
%&=  \inner{f}{T_{-\be_2^*} M_{\be_1^*} f} \inner{h_i}{M_{-\be_1^*} T_{\be_2^*} g_i}\\
&= \inner{M_{-\be_1^*} T_{\be_2^*} f_1}{f_2} \inner{h_i}{M_{-\be_1^*} T_{\be_2^*} g_i} \\
&= \inner{\pi(\la^\circ) f_1}{f_2} \inner{h_i}{\pi(\la^\circ) g_i}\,,
\end{align*}
where $\la^\circ = J \be^* \in J \La^* = \La^\circ$. We can then express $\mathcal{N}_i(z)$ as
\begin{equation*}
\mathcal{N}_i(z) =  \text{vol}(\La)^{-1} \sum_{\la^\circ \in\La^\circ} \inner{\pi(\la^\circ) f_1}{f_2} \inner{h_i}{\pi(\la^\circ) g_i} \,e^{-2\pi i \sigma(\la^\circ,z)}\,,
\end{equation*}
and we have
\begin{equation}\label{eq:series N}
\mathcal{N}(z) = \sum_{i=1}^M \mathcal{N}_i(z) = \text{vol}(\La)^{-1} \sum_{\la^\circ\in\La^\circ} \left (\sum_{i=1}^M \inner{h_i}{\pi(\la^\circ) g_i} \inner{\pi(\la^\circ) f_1}{f_2}\right )\, e^{-2\pi i \sigma(\la^\circ,z)}\,.
\end{equation}
The function $\mathcal{N}$ is continuous since it is a finite sum of continuous functions. By the uniqueness of the Fourier series, \eqref{eq:series} is the Fourier series of $\mathcal{N}$.
\end{proof}

We are now in the position to state the three main identities in Gabor analysis.

\begin{theorem}
Let $\La$ be a lattice in $\Rtd$ with adjoint lattice $\La^\circ$. Then for $g_i,h_i\in M^1(\R^d)$, $i=1,\ldots,M$, the following hold.
\begin{itemize}
\item[i)] Fundamental Identity of Gabor Analysis:
\begin{equation}\label{eq:FIGA}
\sum_{i=1}^M \sum_{\la\in\La} \inner{f_1}{\pi(\la)g_i} \inner{\pi(\la)h_i}{f_2} =   \text{vol}(\La)^{-1} \sum_{i=1}^M \sum_{\la^\circ\in\La^\circ} \inner{h_i}{\pi(\la^\circ)g_i} \inner{\pi(\la^\circ)f_1}{f_2}
\end{equation}
for all $f_1,f_2\in \Lt(\R^d)$.
\item[ii)] Janssen Representation:
\begin{equation}\label{eq:Janssen}
S_{g,h}^\La = \text{vol}(\La)^{-1} \sum_{i=1}^M \sum_{\la^\circ\in\La^\circ} \inner{h_i}{\pi(\la^\circ) g_i} \pi(\la^\circ)\,,
\end{equation}\
where the series converges unconditionally in the strong operator sense.
\item[iii)] Wexler-Raz Biorthogonality Relations:
\begin{equation}\label{eq:WR}
S_{g,h}^\La = I \mbox{  on $\Lt(\R^d)$} \quad   \Longleftrightarrow \quad \text{vol}(\La)^{-1} \sum_{i=1}^M \inner{h_i}{\pi(\la^\circ) g_i} = \delta_{\la^\circ,0} \quad \mbox{for $\la^\circ\in\La^\circ$}\,.
\end{equation}
\end{itemize}
\end{theorem}

\begin{proof}
The Fundamental Identity of Gabor Analysis and Janssen representation follow directly from Proposition~\ref{prop:function N lattice}. It suffices to prove the statements on a dense subspace of $\Lt(\R^d)$. Let $f_1,f_2\in M^1(\R^d)$, then the left hand side of \eqref{eq:FIGA} equals the function $\mathcal{N}$ of Proposition~\ref{prop:function N lattice} evaluated at $z=0$. Since $\mathcal{N}$ equals its Fourier series expansion, we have 
\begin{equation*}
\mathcal{N}(0) =  \text{vol}(\La)^{-1} \sum_{\la^\circ\in\La^\circ} \sum_{i=1}^M \inner{h_i}{\pi(\la^\circ) g_i} \inner{ \pi(\la^\circ) f_1}{f_2}\,,
\end{equation*}
which gives \eqref{eq:FIGA}. 

As for Janssen representation, we observed in the proof of Proposition~\ref{prop:function N lattice} that $\mathcal{N}(z) = \inner{S_{g,h}^\La \pi(z)f_1}{\pi(z) f_2}$ for fixed $f_1,f_2\in M^1(\R^d)$. Evaluating $\mathcal{N}$ at $z=0$ and using the Fourier series representation of $\mathcal{N}$, we obtain
\begin{align*}
\inner{S_{g,h}^\La f_1}{f_2} &=  \text{vol}(\La)^{-1} \sum_{i=1}^M \sum_{\la^\circ\in\La^\circ} \inner{h_i}{\pi(\la^\circ) g_i} \inner{\pi(\la^\circ) f_1}{f_2}\\
&= \innerBig{\text{vol}(\La)^{-1} \sum_{i=1}^M \sum_{\la^\circ\in\La^\circ} \inner{h_i}{\pi(\la^\circ) g_i} \pi(\la^\circ) f_1}{f_2}\,,
\end{align*}
which is the Janssen representation of the frame operator $S_{g,h}^\La$.

The implication $\Longleftarrow$ of $iii)$ follows trivially from the Janssen representation of $S_{g,h}^\La$. For the converse, assume that $S_{g,h}^\La = I$. Let $f_1,f_2 \in M^1(\R^d)$, then $\mathcal{N}$ of Proposition~\ref{prop:function N lattice} is a constant function. Indeed, 
\begin{equation*}
\inner{f_1}{f_2} = \inner{\pi(z)f_1}{\pi(z)f_2} = \inner{S_{g,h}^\La \pi(z)f_1}{\pi(z)f_2} = \mathcal{N}(z)
\end{equation*} 
for every $z\in\Rtd$. Since $\mathcal{N}$ is a constant function, it equals its $0$th Fourier coefficient $\widehat{\mathcal{N}}(0) = \inner{f_1}{f_2}$. By Proposition~\ref{prop:function N lattice}, we have
\begin{equation}\label{eq:rel1}
\text{vol}(\La)^{-1} \sum_{i=1}^M \inner{h_i}{\pi(\la^\circ) g_i} \inner{\pi(\la^\circ) f_1}{f_2} = \delta_{\la^\circ,0} \inner{f_1}{f_2}\,.
\end{equation}
Fix $\la^\circ = (\la_1^\circ,\la_2^\circ)\in \La^\circ$ and let $f\in M^1(\R^d)$ be a nonzero function. By letting $f_1 = T_{-\la_1^\circ} M_{-\la_2^\circ} f$ and $f_2=f$, \eqref{eq:rel1} becomes 
\begin{equation*}
\text{vol}(\La)^{-1} \sum_{i=1}^M \inner{h_i}{\pi(\la^\circ) g_i} \inner{f}{f} = \delta_{\la^\circ,0} \inner{T_{-\la_1^\circ} M_{-\la_2^\circ}f}{f}\,,
\end{equation*}
and the right hand side of \eqref{eq:WR} holds.
\end{proof}

In the subsequent sections we will turn our attention to Gabor frames for model sets by generalizing the construction we have just presented. 

%%%%%%%%%%%%%%%%%%%%%%%%%%%%%%%%%%%%%%%%%%%%%%%%%%%%%%%%%%%%

\section{Almost periodic functions, model sets and local functions}\label{sec:ms}

The main object of our investigation are Gabor frames for model sets, and in the treatment of such frames, we naturally come across almost periodic functions. Therefore, we begin with the review of some basic facts about almost periodic functions and finish with a definition and some properties of model sets. For a detailed exposition on almost periodic functions we refer to \cite{B47, Bes55,BB31}.

We say that a bounded and continuous function $f:\Rtd\rightarrow \C$ is {\it almost periodic}, if to every $\epsilon >0$ there corresponds a relatively dense set $E(f,\epsilon) \subseteq \Rtd$, such that for every $\tau\in E(f,\epsilon)$,
\begin{equation*}
\sup_{z\in\Rtd}\abs{f(z + \tau) - f(z)} \leq \epsilon\,.
\end{equation*}
A subset $D$ is called relatively dense in $\Rtd$ when there exists $r>0$, such that for all $z\in\Rtd$, $D\cap B(z,r) \neq \emptyset$, where $B(z,r)$ is a ball of radius $r$ in $\Rtd$ centered at $z$. Each $\tau\in E(f,\epsilon)$ is called an $\epsilon$-period of $f$. Let $AP(\Rtd)$ denote the space of almost periodic functions. Each almost periodic function is uniformly continuous and admits a formal Fourier series
\begin{equation}\label{eq:Fourier series}
f(z) \sim \sum_{\la \in \sigma(f)} a(\la,f) e^{-2\pi i \la \cdot  z}\,,
\end{equation}
where 
\begin{equation*}
a(\la,f) := \M_z\{f(z) e^{2\pi i \la \cdot z}\} =  \lim_{R\rightarrow \infty} \frac{1}{R^{2d}} \int_{B(0,R)} f(z) e^{2\pi i \la \cdot z}\,dz\,,
\end{equation*} 
are the Fourier coefficients of $f$, and $\sigma(f)$ is the so-called Bohr spectrum of $f$,
\begin{equation*}
\sigma(f) := \big\{\lambda\in\R^m \,:\, a(\la,f) \neq 0 \big \}\,,
\end{equation*}
and it forms a countable set, that is only for a countable number of $\la\in\R^m$, $a(\la,f)$ is nonzero (\cite{B47}). The $0$-th Fourier coefficient of $f$,
\begin{equation*}
a(0,f) := \M_z\{f(z)\}\,,
\end{equation*}
is called {\it the mean value} of $f$. If $f(z,\tilde{z})$, with $(z,\tilde{z}) \in \Rtd \times \Rtd$, is almost periodic, then it is almost periodic with respect to each of the variables $z$ and $\tilde{z}$. Moreover, $\M_{\tilde{z}}\{ f(z,\tilde{z})\}$ is an almost periodic function of $z$.

We gather the important identities of almost periodic functions, that we will use throughout the article, in the following theorem.
\begin{theorem}[\cite{B47}]
Let $f\in AP(\Rtd)$ with the Fourier series given by \eqref{eq:Fourier series}. Then the following hold.
\begin{itemize}
\item[i)] Bohr's Fundamental Theorem: $\M_z\{\abs{f(z)}^2\} = \sum_{\la\in\sigma(f)} \abs{a(\la,f)}^2$.
\item[ii)] Plancherel's Theorem: $\M_z\{ f(z)\, \overline{g(z)} \} = \sum_{\la\in\sigma(f)} a(\la,f)\,\overline{a(\la,g)}$ for all $g\in AP(\Rtd)$ with $\sigma(f) = \sigma(g)$.
\item[iii)] If all the coefficients $a(\la,f)$ of $f\in AP(\R^m)$ are zero, then the function $f\equiv 0$.
\item[iv)] When $f$ is non-negative, then $\M \{ f \} = 0$ if and only if $f\equiv 0$.
\end{itemize}
\end{theorem}

Throughout the exposition we will be encountering almost periodic functions whose spectrums lie in model set. We state the basic definitions and theorems for even dimensional model sets since only those we will use, however the same definitions and properties apply in any dimension.

We begin with a lattice in $\Ga \subset \Rtd \times \R^n$, where $\Rtd$ and $\R^n$ are equipped with Euclidean metrics and $\Rtd \times \R^n$ is the orthogonal sum of the two spaces. Let $p_1:\Rtd \times \R^n \rightarrow \Rtd$ and $p_2: \Rtd \times \R^n \rightarrow \R^n$ be projection maps such that $p_1|\Ga$ is injective and $L= p_1(\Ga)$ is a dense subgroup of $\Rtd$. We impose the same properties on $p_2$. For the dual lattice of $\Ga$, denoted by $\Ga^\ast$, let $p_1^*,p_2^*$ be defined as $p_1,p_2$. It holds then, that  $p_1^*|\Ga^\ast$ is injective and $p_1^*(\Ga^\ast)$ is a dense subgroup of $\Rtd$, and the same holds for $p_2^*$.  Moreover, for $\ga\in\Ga$ and $\ga^*\in\Ga^*$,
\begin{equation*}
\Z \ni \ga \cdot \ga^* = (p_1(\ga),p_2(\ga))\cdot (p_1^*(\ga^*),p_2^*(\ga^*)) = p_1(\ga) \cdot p_1^*(\ga^*) + p_2(\ga)\cdot p_2^*(\ga^*)
\end{equation*}
Let $\Om\subset \R^n$ be compact, equal to the closure of its interior and have boundary of measure $0$. We call $\Om$ a window. Then a model set $\La(\Om)$ is defined as 
\begin{equation*}
\La(\Om) :=  \left \{p_1(\ga)\,:\, \ga\in\Ga,\, p_2(\ga) \in \Om \right \}  \subset L \subset \Rtd\,.
\end{equation*}
If $\Om$ is symmetric around the origin then $0\in\La(\Om)$. Model set is {\it generic} if the boundary of $\Om$ has no common points with $p_2(\Ga)$. A model set is {\it simple} if $n=1$. We will be working only with simple model sets. We assume, without loss of generality, that from now on $\Om$ is symmetric around the origin.

Model sets are a very natural generalizations of lattices, and for $n=0$ they reduce to a lattice and, thus, the results that we develop later on in the article reduce to the known ones for lattices. If $\La(\Om)$ is a model set, then it is uniformly discrete, relatively dense, and has a well defined density 
\begin{equation*}
D(\La(\Om)) = \lim_{R\rightarrow \infty} \frac{\# \{\La(\Om)\cap B(x,R)\}}{R^{2d}}\,,
\end{equation*}
where $\#S$ denotes the cardinality of the set $S$ and $B(x,R)$ is a ball of radius $R$ in $\Rtd$ centered at $x$ . The limit is independent of $x\in\Rtd$. For a simple model set $\La(\Om)$, we have $D(\La(\Om)) = \text{vol}(\Ga)^{-1} \abs{\Om}$, see \cite{BM00}.

Due to the underlying lattice structure of a model set, there exists a Poisson summation formula for $\La(\Om)$. Let $C_0^\infty(\Om)$ be the space of all smooth, real valued functions on $\R$ with support in $\Om$. Via the mapping $p_2\circ(p_1|_\Ga)^{-1}: L \rightarrow \R$ we obtain a space $\mathcal{C}(\La(\Om))$ of functions on $L$, vanishing off $\La(\Om)$:  for $\psi \in C_0^\infty(\Om)$, we define $w_\psi \in \mathcal{C}(\La(\Om))$ by
\begin{equation}\label{def:w_psi}
w_\psi: L \rightarrow \R \,,\quad \quad w_\psi(\la) := \psi(p_2(\ga)) \quad \mbox{for} \quad \la=p_1(\ga)\in \La(\Om)\,,
\end{equation}
and $w_\psi(\la)=0$ for $\la\notin \La(\Om)$. If $\psi$ were the indicator function of $\Om$, we would have $w_\psi(\la) = 1$ on $\La(\Om)$ and $w_\psi(\la)=0$ if $\la \notin \La(\Om)$. However, the indicator function is not smooth. The Poisson summation formula for model sets was originally stated for the class of Schwartz functions in \cite{Me12}. However, since it relies on the original Poisson summation formula, we can state it for a bigger space.

\begin{theorem}[Poisson Summation Formula for Model Sets]\label{thm:poisson_MS}
Let $\La(\Om)$ be a simple model set defined by a relatively compact set $\Om\subseteq \R$ of non-empty interior and a lattice $\Ga\subseteq \Rtd \times \R$. Let $\psi\in C_0^\infty(\Om)$, and the weight factors $w_\psi(\la)$ on $\La(\Om)$ be defined as in \eqref{def:w_psi}. Then, for every $F\in M^1(\Rtd)$, the following holds
\begin{equation}\label{eq:poisson}
\sum_{\la\in\La(\Om)} w_\psi(\la) F(\la)e^{-2\pi i \la \cdot z} =  \sum_{\ga^\ast \in \Ga^\ast} \widetilde{w}_\psi(-p^*_2(\ga^\ast))  \widehat{F}(z - p^*_1(\ga^\ast))\,,
\end{equation}
where 
\begin{equation}\label{eq:Fourier weight}
\widetilde{w}_\psi(p^*_2(\ga^\ast)) := \text{vol}(\Ga)^{-1} \widehat{\psi}(p^*_2(\ga^\ast))\quad \quad \mbox{for} \quad \ga^\ast\in \Ga^\ast\,.
\end{equation}
The identity holds pointwise for all $z\in\Rtd$, and both sums converge uniformly and absolutely for all $t\in\Rtd$. \end{theorem}

%\begin{remark}
Meyer, in \cite{Me12}, originally stated the Poisson summation formula for model sets for functions in the Schwartz class $\mathcal{S}(\Rtd)$. As the Poisson summation formula for model sets follows from the ordinary Poisson summation formula, which holds not only for the elements from $\mathcal{S}(\Rtd)$ but also for functions in $M^1(\Rtd)$, we were able to extend the former one to a bigger class of functions.
%\end{remark}

\begin{proof}
Since $w_\psi(\la) = \psi(p_2(\ga))$ for $\la=p_1(\ga)$, one can forget about the restriction $\la\in\La$ which is given for free by the support of $\psi$, and consider a function $f=F \otimes \psi$. Since  $\psi \in C_0^\infty(\Om) \subset M^1(\R)$, $f\in M^1(\Rtd \times \R)$ by the tensor product property of Feichtinger's algebra $M^1$. We can now apply the ordinary Poisson summation formula, Theorem~\ref{thm:poisson}, to the lattice $\Ga$, its dual lattice $\Ga^\ast$ and the function $f$ and obtain 
\begin{equation*}
\sum_{\ga \in \Ga} f(\ga) e^{-2\pi i \ga \cdot z} = \text{vol}(\Ga)^{-1} \sum_{\ga^*\in \Ga^*} \F f(z-\ga^*)\,, 
\end{equation*}
Then \eqref{eq:poisson} follows by taking $z=(x,0)\in \Rtd\times \R$.
\end{proof}

\begin{remark}\label{rem:a.p. construction}
Poisson summation formula for model sets gives a method for constructing almost periodic functions with desired spectrum. Indeed, the function on the right hand side of \eqref{eq:poisson} is almost periodic since it equals an absolutely convergent trigonometric series. By the property of almost periodic functions, the Fourier series of this function coincides with this trigonometric series. That means that the Fourier coefficients of the right hand side of \eqref{eq:poisson} equal $w_\psi(\la)F(\la)$.
\end{remark}

On the collection of point sets in $\Rtd$ that are relatively dense and uniformly separated, with minimal separation greater than $r$, denoted by $D_r(\Rtd)$, we can put a topology, called local topology: two sets $\La$ and $\La'$ of $D_r(\Rtd)$ are close if, for some large $R$ and some small $\epsilon$, one has
\begin{equation}\label{eq:topology}
\La'\cap B(0,R) = (\La +v)\cap B(0,R) \quad \text{for some $v\in B(0,\epsilon)$.}
\end{equation}
Thus for each point of $\La$ within the ball $B(0,R)$, there is a point of $\La'$ within the distance $\epsilon$ of that point, and vice versa. The pairs $(\La,\La')$ satisfying \eqref{eq:topology} are called $(R,\epsilon)$-close. More formally, for $\epsilon >0$ and a ball $B(x,R)$, define
\begin{equation*}
U(\epsilon,B(x,R)):= \{(\La,\La')\in D_r(\Rtd)\times D_r(\Rtd):(\La +v)\cap B(x,R) = \La '\cap B(x,R), \mbox{for some}\, v\in B(0,\epsilon)\}\,.
\end{equation*}
These sets form a fundamental system for a uniform structure on $D_r(\Rtd)$ whose topology has the sets
\begin{equation*}
U(\epsilon,B(x,R))[\La] := \{ \La' \in D_r(\Rtd):(\La,\La')\in U(\epsilon,B(x,R))\}
\end{equation*}
as a neighbourhood basis of $\La$. Note, all the point sets $\La$ from $D_r(\Rtd)$ have the same relative separation $\text{rel}(\La)$.

On the set $D_r(\Rtd)$ we can put a metric. Let $\La,\La' \in D_r(\Rtd)$, then
\begin{equation*}
d(\La,\La') := \limsup_{R\rightarrow \infty} \frac{\# \{ ((\La \cup \La')\setminus (\La \cap \La')) \cap B(0,R)\}}{R^{2d}}
\end{equation*}
is a pseudometric on $D_r(\Rtd)$. We obtain a metric by defining the equivalence relation
\begin{equation*}
\La \equiv \La' \iff d(\La,\La')=0\,.
\end{equation*}

Later in the article, we will work with a collection of model sets. Let $\Om$ be a window, then for each $(s,t)\in\Rtd\times\R$ we may define
\begin{equation*}
\La_{(s,t)}^\Om = s+ \La(\Om - t)
\end{equation*}
Note that $\La(\Om)$ and all its shifts have the same relative separation $\text{rel}(\La(\Om))$.

If $(s,t)\equiv(s',t') \mod \Ga$, then $\La_{(s,t)}^\Om=\La_{(s',t')}^\Om$, however the inverse is not necessarily true. In the sequel we will write $(s,t)_L$ for the congruence class $(s,t)\mod \Ga$. These model sets are parametrized  by the torus  $\T:=(\Rtd\times \R) / \Ga = (\R/\Z)^{2d+1}$. There is  a natural measure, Haar measure, $\theta$ on $\T$. It is invariant under the action of $\Rtd$ on $(\Rtd\times \R )/\Ga$ and it acts by
\begin{equation*}
z+(s,t)_L = (z+s,t)_L\,.
\end{equation*}
We can define an embedding $\Rtd \rightarrow \T$, $z\mapsto (z,0)_L$. The image of this embedding is dense in $\T$.

Now, let $\La(\Om)$ be a model set, and we translate it by elements $z\in\Rtd$
\begin{equation*}
z+\La(\Om) = z + \La(\Om+0) = \La_{(z,0)}^\Om\,.
\end{equation*}
The closure of the set of all translates $\La_{(z,0)}^\Om$ of $\La(\Om)$ under the local topology \eqref{eq:topology} forms the so-called local hull $X(\La(\Om))$ of $\La(\Om)$, $X(\La(\Om)) = \overline{\{z+\La(\Om):z\in\Rtd\}}$, (\cite{M05},\cite{Sch00}).

\begin{proposition}\cite{Sch00}
Let $\La(\Om)$ be a model set. There is a continuous mapping 
\begin{equation*}
\beta :\, X(\La(\Om)) \rightarrow \T\,,
\end{equation*}
called the torus parametrization, such that $(i)$ $\beta$ is onto; $(ii)$ $\beta$ is injective almost everywhere with respect to the Haar measure $\theta$; $(iii)$ $\beta(z+\La') = z+ \beta(\La')$ for all $z\in\Rtd$ and all $\La'\in X(\La)$; and $(iv)$ $\beta(z+\La(\Om)) = (z,0)_L$ for all $z\in\Rtd$.
\end{proposition}
By injective almost everywhere, we mean that the set $P$ of points $z\in\T$, for which there is more than one point set of $X(\la(\Om))$ over $z$, satisfies $\theta(P)=0$.

There is a unique $\Rtd$-invariant measure $\mu$ on $X(\La(\Om))$, with $\mu(X(\La(\Om))) = 1$, and $\beta$ relates the Haar measure $\theta$ and $\mu$ through: $\theta(P) = \mu(\beta^{-1} P)$ for all measurable subsets $P$ of $\T$. Having $\mu$ we can introduce the space $\Lt(X(\La(\Om)),\mu)$ of square integrable functions on $X(\La(\Om))$. Square integrable functions on $X(\La(\Om))$ and square integrable functions on $\T$ can be identified,
\begin{equation}\label{eq:iso}
\Lt(X(\La(\Om)),\mu) \simeq \Lt(\T,\theta)\,.
\end{equation}
The mapping takes a function $\mathfrak{N}\in \Lt(\T,\theta)$ and creates $\widetilde{\mathcal{N}}=\mathfrak{N} \circ \beta \in \Lt(X(\La(\Om)),\mu)$, and since $\beta$ is almost everywhere injective, the map is a bijection. This allows us to analyze functions on $X(\La(\Om))$ by treating them as functions on $\T$.

Consider a function $\widetilde{\mathcal{N}}:X(\La(\Om)) \rightarrow \C$. We can define from it a function $\mathcal{N}:\Rtd \rightarrow \C$ by 
\begin{equation*}
\mathcal{N}(z) = \widetilde{\mathcal{N}}(z+\La(\Om))\,.
\end{equation*}
If $\widetilde{\mathcal{N}}$ is continuous, then for all $z_1,z_2\in\Rtd$, if $z_1+\La(\Om)$ and $z_2+\La(\Om)$ are close, then $\widetilde{\mathcal{N}}(z_1+\La(\Om))$ and $\widetilde{\mathcal{N}}(z_2+\La(\Om))$ are close, and as a consequence, $\mathcal{N}(z_1)$ and $\mathcal{N}(z_2)$ are close. Thus continuity of $\widetilde{\mathcal{N}}$ implies continuity of $\mathcal{N}$, or a certain locality. More formally, a function $\mathcal{N}:\Rtd\rightarrow \C$ is called {\it local with respect to} $\La(\Om)$, if for all $\delta>0$ there exists $R>0$ and $\epsilon>0$ so that whenever $z_1+\La(\Om)$ and $z_2+\La(\Om)$, for $z_1,z_2\in\Rtd$, are $(R,\epsilon)$-close, then
\begin{equation*}
\abs{\mathcal{N}(z_1)-\mathcal{N}(z_2)} < \delta\,.
\end{equation*}  
Intuitively, $\mathcal{N}$ looks very much the same at places where the local environment looks the same. It can be easily verified that local functions are continuous on $\Rtd$ and almost periodic.

Using locality, we can go in the opposite direction. Let $\mathcal{N}$ be a local function with respect to $\La(\Om)$. Define a function $\widetilde{\mathcal{N}}$ on the orbit of $\La(\Om)$:
\begin{equation*}
\widetilde{\mathcal{N}}:\{ z+\La(\Om):z\in\Rtd\} \rightarrow \C \quad \mbox{by} \quad \widetilde{\mathcal{N}}(z+\La(\Om)) = \mathcal{N}(z)\,.
\end{equation*}
Then $\widetilde{\mathcal{N}}$ is uniformly continuous on $\{ z+\La(\Om):z\in\Rtd\}$ with respect to the local topology. The reason for this is that the continuity condition which defines the localness of $\mathcal{N}$ is based on the uniformity defining the local topology on $\{ z+\La(\Om):z\in\Rtd\}$. It follows that $\widetilde{\mathcal{N}}$ lifts uniquely to a continuous function on a local hull $X(\La(\Om))$.

\begin{proposition}\cite{MNP08}
For each local function $\mathcal{N}$ with respect to $\La(\Om)$ there is a unique continuous function $\widetilde{\mathcal{N}}$ on a local hull $X(\La(\Om))$, whose restriction to the orbit of $\La(\Om)$ is $\mathcal{N}$. Every continuous function on the local hull of $\La(\Om)$ arises in this way.
\end{proposition}

The spectral theory of $\Lt(X(\La(\Om)),\mu)$ allows us to analyze $\mathcal{N}$ by analyzing its corresponding function $\widetilde{\mathcal{N}}$ on $\Lt(X(\La(\Om)),\mu)$. Suppose $\mathcal{N}$ is a local function with respect to the model set $\La(\Om)$. From the locality of $\mathcal{N}$ we have its extension $\widetilde{\mathcal{N}}\in\Lt(X(\La(\Om)),\mu)$ which is continuous. Then we obtain $\mathfrak{N}\in \Lt(\T,\theta)$, where
\begin{equation*}
\mathfrak{N}((z,0)_L) = \mathfrak{N}(\beta(z+\La(\Om))) = \widetilde{\mathcal{N}}(z+\La(\Om)) = \mathcal{N}(z)\,,
\end{equation*}
and since functions in $\Lt(\T,\theta)$ have Fourier expansions, we can write
\begin{equation}\label{eq:Fourier series N}
\mathcal{N}(z) = \widetilde{\mathcal{N}}(z+\La(\Om)) = \mathfrak{N}((z,0)_L) = \sum_{\eta\in\Ga^*} \widehat{\mathfrak{N}}(\eta) e^{-2\pi i (z,0) \cdot \eta} = \sum_{\eta\in\Ga^*} \widehat{\mathfrak{N}}(\eta) e^{-2\pi i z\cdot p_1^*(\eta)}\,,
\end{equation}
almost everywhere, with 
\begin{equation*}
\widehat{\mathfrak{N}}(\eta) = \int_{\T} \mathfrak{N}((s,t)_L) e^{2\pi i (s,t) \cdot \eta}\, d\theta(s,t)\,.
\end{equation*}
However, we know $\mathfrak{N}$ only on $(\Rtd,0)_L$. To compute the coefficients $\widehat{\mathfrak{N}}(\eta)$ out of $\mathcal{N}$ alone, we can use the Birkhoff ergodic theorem
\begin{align*}
\widehat{\mathfrak{N}}(\eta)  &= \int_{\T} \mathfrak{N}((s,t)_L) e^{2\pi i (s,t) \cdot \eta}\, d\theta(s,t)= \lim_{R\rightarrow \infty} \frac{1}{R^{2d}} \int_{B(0,R)} \mathfrak{N}((z,0)_L) e^{2\pi i (z,0) \cdot \eta}\, dz\\
&= \lim_{R\rightarrow \infty} \frac{1}{R^{2d}} \int_{B(0,R)} \mathcal{N}(z) e^{2\pi i z \cdot p_1^*(\eta)}\, dz\,,
\end{align*}
where we used $\mathfrak{N}((z,0)_L) = \mathcal{N}(z)$ and $\eta = (p_1^*(\eta),p_2^*(\eta))$, so
\begin{equation*}
(z,0)\cdot \eta = z \cdot p_1^*(\eta) + 0 \cdot p_2^*(\eta)\,.
\end{equation*}
If $ \sum_{\eta\in\Ga^*} \abs{\widehat{\mathfrak{N}}(\eta)}<\infty$, then the Fourier series \eqref{eq:Fourier series N} converges absolutely to $\mathcal{N}(z)$ for all $z\in\Rtd$.

%%%%%%%%%%%%%%%%%%%%%%%%%%%%%%%%%%%%%%%%%%%%%%%%%%%%%%%%%%%%

\section{Bracket product on model sets}\label{sec:bracket}

As described in the introduction, we are interested in the charaterization of tight and dual Gabor frames for simple model sets. We are going to imitate the approach presented in Section~\ref{sec:gabor frames lattice} for model sets, and like in the previous Section, the Poisson summation formula will play a crucial role.

We assume from now on that $\Om$ is symmetric around the origin and that $p_2(\Ga)$ and $p_2^*(\Ga^\ast)$ have no common points with the boundary of $\Om$. Let $\La(\Om)$ be a simple model set and $\psi\in C_0^\infty(\Om)$. Let $\widetilde{w}_\psi$ be a function defined as in Theorem~\ref{thm:poisson_MS}. Then for a fixed $z=(x,\om)\in\Rtd$ the generalized $\psi${\it -bracket product} of $f$ and $g$ is defined as
\begin{equation}\label{def:bracket}
\big [\widehat{\pi(z)f},\widehat{g}\big ]_{\La(\Om)}^{\psi} (\tilde{z}) := \sum_{\ga^\ast \in \Ga^\ast} \widetilde{w}_\psi(-p^*_2(\ga^\ast)) M_{-z}\W(\widehat{f}, \widehat{g})(\tilde{z} - p^*_1(\ga^\ast))\,.
\end{equation}
For $f,g\in M^1(\R^d)$, we have $\W(\widehat{f},\widehat{g}) \in M^1(\Rtd)$ and the bracket product is well defined. Moreover, $\F^{-1} M_{-z}\W(\widehat{f}, \widehat{g}) = T_z \A(f,g)$ and is also an element of $M^1(\Rtd)$, and by Remark~\ref{rem:a.p. construction}, $\big [\widehat{\pi(z)f},\widehat{g}\big ]_{\La(\Om)}^{\psi}$ is an almost periodic function represented by the trigonometric series
\begin{equation*}
\big [\widehat{\pi(z)f},\widehat{g}\big ]_{\La(\Om)}^{\psi}(\tilde{z}) = \sum_{\la\in\La} w_\psi(\la) \A(f,g)(\la-z) e^{-2\pi i \la \cdot \tilde{z}}\,.
\end{equation*}
The Fourier coefficients are given by
\begin{equation}\label{eq:Fourier coeff}
\M_{\tilde{z}}\Big \{\big [\widehat{\pi(z)f},\widehat{g}\big ]_{\La(\Om)}^{\psi}(\tilde{z}) e^{2\pi i \la \cdot \tilde{z}} \Big \} = w_{\psi}(\la) \A(f,g)(\la - z)\,.
\end{equation}

We make the following useful observation that is in analogy with regular shifts.

\begin{lemma}\label{lem:lemma2}
Let $\La(\Om)$ be a simple model set and $\psi\in C_0^\infty(\Om)$. For all functions $f_1,f_2,g,h\in M^1(\R^d)$, we have
\begin{equation*}
\sum_{\la\in\La(\Om)} w_\psi(\la)^2 \, \A(f_1,g)(\la-z)\overline{\A(f_2,h)(\la-z)} = \M\Big\{ \big [\widehat{\pi(z)f_1},\widehat{g}\big ]_{\La(\Om)}^{\psi} \cdot \overline{\big [\widehat{\pi(z)f_2},\widehat{h}\big ]_{\La(\Om)}^{\psi} }\Big\}\,.
\end{equation*}
\end{lemma}

\begin{proof}
The bracket products $\big [\widehat{\pi(z)f_1},\widehat{g}\big ]_{\La(\Om)}^{\psi}$ and $\big [\widehat{\pi(z)f_2},\widehat{h}\big ]_{\La(\Om)}^{\psi}$ are almost periodic function with Fourier coefficients given by $w_{\psi}(\la) \A(f_1,g)(\la - z)$ and $w_{\psi}(\la) \A(f_2,h)(\la - z)$, respectively. Using Plancherel Theorem for Fourier series of almost periodic functions and \eqref{eq:Fourier coeff}, we obtain 
\begin{align*}
\sum_{\la\in\La(\Om)} w_\psi(\la)^2\, &\A(f_1,g)(\la-z)\overline{\A(f_2,h)(\la-z)} \\
&= \sum_{\la\in\La(\Om)} \M_{\tilde{z}}\Big \{\big [\widehat{\pi(z)f_1},\widehat{g}\big ]_{\La(\Om)}^{\psi}(\tilde{z}) e^{2\pi i \la \cdot \tilde{z}} \Big \} \, \overline{\M_{\tilde{z}}\Big \{\big [\widehat{\pi(z)f_2},\widehat{h}\big ]_{\La(\Om)}^{\psi}(\tilde{z}) e^{2\pi i \la \cdot \tilde{z}} \Big \} }\\
&=  \M\Big\{ \big [\widehat{\pi(z)f_1},\widehat{g}\big ]_{\La(\Om)}^{\psi} \cdot \overline{\big [\widehat{\pi(z)f_2},\widehat{h}\big ]_{\La(\Om)}^{\psi} }\Big\}\,.
\end{align*}
\end{proof}

The following result concerning the bracket product will be important in many calculations to follow.
\begin{proposition}\label{prop:ap bracket}
Let $\La(\Om)$ be a simple model set and $\psi\in C_0^\infty(\R)$ be non-negative. Assume that $g,h\in M^1(\R^d)$. Then, for $f_1,f_2\in M^1(\R^d)$,
\begin{equation*}
F(z,\tilde{z}) =  \big [\widehat{\pi(z)f_1},\widehat{g}\big ]_{\La(\Om)}^{\psi}(\tilde{z}) \cdot \overline{\big [\widehat{\pi(z)f_2},\widehat{h}\big ]_{\La(\Om)}^{\psi}(\tilde{z})}, \quad (z,\tilde{z}) \in \Rtd \times \Rtd
\end{equation*}
is an almost periodic function.
\end{proposition}

\begin{proof}
Let $f_1,f_2\in M^1(\R^d)$. Moreover, for $\eta\in \Ga^\ast$, we define $\Psi_\eta$ such that $\widehat{\Psi_\eta} = \widehat{\psi} \cdot T_{p^*_2(\eta)} \widehat{\psi}$. Then each $\Psi_\eta$ belongs to $C_0^\infty(\R)$ and is compactly supported on $\Om+\Om$, and we can define $\widetilde{w}_{\Psi_\eta}$ as $\widetilde{w}_{\Psi_\eta} (p^*_2(\ga^*))=\text{vol}(\Ga)^{-1} \widehat{\Psi_\eta}(p^*_2(\ga^*))$, for all $\ga^*\in\Ga^*$, as in \eqref{eq:Fourier weight}. Then, by the change of index, we have
\begin{align*}
&F(z,\tilde{z}) =  \sum_{\mu,\theta \in \Ga^\ast} \widetilde{w}_\psi(-p^*_2(\mu) \widetilde{w}_\psi(-p^*_2(\theta)) M_{-z}\W(\widehat{f_1}, \widehat{g})(\tilde{z} - p^*_1(\mu)) \overline{M_{-z}\W(\widehat{f_2}, \widehat{h})(\tilde{z} - p^*_1(\theta))}\\
&=  \sum_{\mu,\theta \in \Ga^\ast} \widetilde{w}_\psi(-p^*_2(\mu)) \widetilde{w}_\psi(-p^*_2(\theta)) \W(\widehat{f_1}, \widehat{g})(\tilde{z} - p^*_1(\mu))  \overline{\W(\widehat{f_2}, \widehat{h})(\tilde{z} - p^*_1(\theta))} e^{-2\pi i (p^*_1(\theta) -  p^*_1(\mu)) \cdot z}\\
&= \sum_{\eta,\mu \in \Ga^\ast} \widetilde{w}_\psi(-p^*_2(\mu)) \widetilde{w}_\psi(-p^*_2(\mu)-p^*_2(\eta)) \W(\widehat{f_1}, \widehat{g})(\tilde{z} - p^*_1(\mu))  \overline{T_{p^*_1(\eta)}\W(\widehat{f_2}, \widehat{h})(\tilde{z} - p^*_1(\mu))} e^{-2\pi i p^*_1(\eta) \cdot z}\\
&= \sum_{\eta,\mu \in \Ga^\ast} \text{vol}^{-2}(\Ga) \widehat{\Psi_\eta}(-p^*_2(\mu)) \W(\widehat{f_1}, \widehat{g})(\tilde{z} - p^*_1(\mu))  \overline{T_{p^*_1(\eta)}\W(\widehat{f_2}, \widehat{h})(\tilde{z} - p^*_1(\mu))} e^{-2\pi i p^*_1(\eta) \cdot z}\\
&=\text{vol}^{-1}(\Ga) \sum_{\eta \in \Ga^\ast} \left [\sum_{\mu \in \Ga^\ast} \widetilde{w}_{\Psi_\eta} (-p^*_2(\mu)) \W(\widehat{f_1}, \widehat{g})(\tilde{z} - p^*_1(\mu))  \overline{T_{p^*_1(\eta)}\W(\widehat{f_2}, \widehat{h})(\tilde{z} - p^*_1(\mu))}\right ] e^{-2\pi i p^*_1(\eta) \cdot z}\,.
\end{align*}
For a fixed $\eta\in\Ga^\ast$, consider the series
\begin{equation*}
F_\eta(\tilde{z}) = \sum_{\mu \in \Ga^\ast} \widetilde{w}_{\Psi_\eta} (-p^*_2(\mu)) \Big( \W(\widehat{f_1}, \widehat{g}) \cdot \overline{T_{p^*_1(\eta)}\W(\widehat{f_2}, \widehat{h})} \Big)(\tilde{z} - p^*_1(\mu))\,.
\end{equation*}
Since $\W(\widehat{f_1}, \widehat{g})$ and $\W(\widehat{f_2}, \widehat{h})$ lie in $M^1(\Rtd)$, their product as well. Moreover, 
\begin{equation*}
\F^{-1}\big (\W(\widehat{f_1}, \widehat{g}) \overline{T_{p^*_1(\eta)}\W(\widehat{f_2}, \widehat{h})}\big ) = \A(f_1,g) \ast (M_{p_1^*(\eta)} \A(f_2,h))^*\, %= \A(f_1,g) \ast M_{p_1^*(\eta)} \A(h,f_2)\,,
\end{equation*}
which is in $M^1(\Rtd)$ since $\A(f_1,g),\A(f_2,h) \in M^1(\Rtd)$. Therefore, by Poisson summation formula for model sets we can write $F_\eta(\tilde{z})$ as
\begin{align*}
F_\eta(\tilde{z}) &= \sum_{\ga\in\Ga} \Psi_\eta(p_2(\ga))\big (\A(f_1,g) \ast  (M_{p_1^*(\eta)} \A(f_2,h))^*\big ) \big (p_1(\ga)\big )\, e^{-2\pi i p_1(\ga) \cdot \tilde{z}}\\
&= \sum_{\ga\in\Ga} \Psi_\eta(p_2(\ga)) \innerBig{\A(f_1,g)}{T_{p_1(\ga)}M_{p_1^*(\eta)} \A(f_2,h)}_{\Lt(\Rtd)}\, e^{-2\pi i p_1(\ga) \cdot \tilde{z}}\,,
\end{align*}
and $F(z,\tilde{z})$ becomes
\begin{align*}
&F(z,\tilde{z}) \\
&= \text{vol}^{-1}(\Ga) \sum_{\eta \in \Ga^\ast} \sum_{\ga\in\Ga} \Psi_\eta(p_2(\ga)) \innerBig{\A(f_1,g)}{T_{p_1(\ga)}M_{p_1^*(\eta)} \A(f_2,h)}_{\Lt(\Rtd)} \, e^{-2\pi i p_1(\ga) \cdot \tilde{z}}\, e^{-2\pi i p^*_1(\eta) \cdot z} \\
&= \text{vol}^{-1}(\Ga) \sum_{\eta \in \Ga^\ast} \sum_{\ga\in\Ga} \inner{\psi}{T_{p_2(\ga)} M_{p_2^*(\eta)} \psi} \innerBig{\A(f_1,g)}{T_{p_1(\ga)}M_{p_1^*(\eta)} \A(f_2,h)}_{\Lt(\Rtd)} \, e^{-2\pi i p_1(\ga) \cdot \tilde{z}}\, e^{-2\pi i p^*_1(\eta) \cdot z} \\
&= \text{vol}^{-1}(\Ga) \sum_{\eta \in \Ga^\ast} \sum_{\ga\in\Ga}  \innerBig{\psi \otimes \A(f_1,g)}{T_\ga M_{\eta} \big (\psi \otimes \A(f_2,h)\big)}_{\Lt(\R\times \Rtd)} \, e^{-2\pi i p_1(\ga) \cdot \tilde{z}}\, e^{-2\pi i p^*_1(\eta) \cdot z}\,.
\end{align*}
The coefficients in the series defining $F(z,\tilde{z})$ are in $\lo(\Ga \times \Ga^\ast)$, because $\psi \otimes \A(f_1,g)$ and $\psi \otimes \A(f_2,h)$ lie in $M^1(\R\times \Rtd)$, and hence 
\begin{equation*}
\V_{\big (\psi \otimes \A(f_2,h)\big)} \big (\psi \otimes \A(f_1,g)\big) \in \mathbf{W}(C_0,\lo)\big((\R\times\Rtd) \times (\R\times\Rtd)\big)\,.
\end{equation*}
That means that $F$ equals a generalized trigonometric polynomial, and therefore is almost periodic.
\end{proof}

We need one more result that will be an important tool in the characterization of tight frames and dual frames.

\begin{proposition}\label{prop:function N}
Let $\La(\Om)$ be a simple model set and $\psi\in C_0^\infty(\Om)$ non-negative function. Assume that $g_i,h_i\in M^1(\R^d)$ for every $i=1,\ldots,M$. Then, for every $f_1,f_2\in M^1(\R^d)$, the function
\begin{equation}\label{eq:function G}
\mathcal{N}^\psi(z) = \sum_{i=1}^M \sum_{\la\in\La(\Om)} w_\psi(\la)^2 \, \inner{\pi(z) f_1}{\pi(\la) g_i} \inner{\pi(\la) h_i}{\pi(z) f_2}
\end{equation}
is almost periodic and coincides pointwise with its Fourier series $\sum_{\eta \in \Ga^\ast} \widehat{\mathcal{N}^\psi}(\eta) e^{-2\pi i p^*_1(\eta) \cdot z}$, with
\begin{equation*}
\widehat{\mathcal{N}^\psi}(\eta) = \text{vol}(\Ga)^{-1} \, \widehat{\psi^2}(-p_2^*(\eta))\, \innerBig{\pi\big (J p_1^*(\eta)\big) f_1}{f_2}\,\sum_{i=1}^M \innerBig{h_i}{\pi\big(J p_1^*(\eta)\big) g_i}\,,
\end{equation*}
where $\eta\in \Ga^\ast$ and $J$ is the symplectic matrix.
\end{proposition}

\begin{proof}
Let $f_1,f_2 \in M^1(\R^d)$. Using the relationship between short time Fourier transform and cross-ambiguity function, we can express $\mathcal{N}^{\psi}$ as
\begin{align*}
\mathcal{N}^\psi (z) &= \sum_{i=1}^M \sum_{\la\in\La(\Om)} w_\psi(\la)^2\, \inner{\pi(z) f_1}{\pi(\la) g_i} \inner{\pi(\la) h_i}{\pi(z) f_2}\\
&= \sum_{i=1}^M \sum_{\la\in\La(\Om)} w_\psi(\la)^2\, \A(f_1,g_i)(\la-z) \overline{\A(f_2,h_i)(\la-z)}\,.
\end{align*}
For $i=1,\ldots,M$ fixed, let
\begin{equation*}
\mathcal{N}^\psi_i (z) = \sum_{\la\in\La(\Om)} w_\psi(\la)^2\, \A(f_1,g_i)(\la-z) \overline{\A(f_2,h_i)(\la-z)}\,.
\end{equation*}
By Lemma~\ref{lem:lemma2}, we have
\begin{equation*}
\mathcal{N}^\psi_i(z) = \M_{\tilde{z}}\Big\{ \big [\widehat{\pi(z)f_1},\widehat{g_i}\big ]_{\La(\Om)}^{\psi}(\tilde{z}) \cdot \overline{\big [\widehat{\pi(z)f_2},\widehat{h_i}\big ]_{\La(\Om)}^{\psi}(\tilde{z}) }\Big\}\,.\,,
\end{equation*}
and by Proposition~\ref{prop:ap bracket} and properties of almost periodic functions of two variables, $\mathcal{N}^\psi_i(z)$ is almost periodic. Moreover, using functions $F_\eta$ and $\Psi_\eta$ defined in the proof of Proposition~\ref{prop:ap bracket}, we have
\begin{align*}
\mathcal{N}^\psi_i(z) &= \text{vol}(\Ga)^{-1} \sum_{\eta\in \Ga^*} \M_{\tilde{z}} \Big \{ F_\eta(\tilde{z})\Big\} e^{-2\pi i z \cdot p_1^*( \eta) } \\
&=  \text{vol}(\Ga)^{-1} \sum_{\eta\in \Ga^*} \Psi_\eta(0) \innerBig{\A(f_1,g_i)}{M_{p_1^*(\eta)} \A(f_2,h_i)}_{\Lt(\Rtd)}\, e^{-2\pi i z \cdot p_1^*(\eta)} \\
&= \text{vol}(\Ga)^{-1} \sum_{\eta\in \Ga^*} \widehat{\psi^2}(-p_2^*(\eta))\, \innerBig{\pi\big (J p_1^*(\eta)\big) f_1}{f_2}\,\innerBig{h_i}{\pi\big(J p_1^*(\eta)\big) g_i}\, e^{-2\pi i z \cdot p_1^*(\eta)} \,,
\end{align*}
where the last equality follows from relations between cross-ambiguity function, cross-Wigner distribution and Moyal's formula, as in the proof of Proposition~\ref{prop:function N lattice}.

Now, since $\mathcal{N}$ is a finite sum of almost periodic functions, it is almost periodic and it equals a generalized trigonometric series
\begin{equation}\label{eq:series}
\mathcal{N}^\psi(z) = \text{vol}(\Ga)^{-1} \sum_{\eta\in \Ga^*}\left ( \widehat{\psi^2}(-p_2^*(\eta))\, \innerBig{\pi\big (J p_1^*(\eta)\big) f_1}{f_2}\,\sum_{i=1}^M \innerBig{h_i}{\pi\big(J p_1^*(\eta)\big) g_i} \right )\, e^{-2\pi i z \cdot p_1^*(\eta)}\,. 
\end{equation}
By the uniqueness of the Fourier series, \eqref{eq:series} is the Fourier series of $\mathcal{N}^\psi(z)$.
\end{proof}

%%%%%%%%%%%%%%%%%%%%%%%%%%%%%%%%%%%%%%%%%%%%%%%%%%%%%%%%%%%%

\section{Gabor Analysis for Model Sets }\label{sec:main}

We first begin with weighted Gabor frames and characterize normalized tight and dual weighted Gabor frames. The characterization follows directly from the bracket product defined in the previous section. Next, we move to the non-weighted scenario, where we develop Fundamental Identity of Gabor Analysis for model sets, Janssen representation and Wexler-Raz orthogonality relations.

%%%%%%%%%%%%%%%%%%%%%%%%%%%%%%%%%%%%%%%%%%%%%%%%%%%%%%%%%%%%

\subsection{Weighted Gabor Systems}

Let $\La(\Om)$ be a simple model set, $\psi\in C_0^\infty(\Om)$ non-negative function and the windows $g_1,\ldots, g_M\in M^1(\R^d)$. Then, a {\it weighted Gabor system}, denoted by $\G_\psi(g_1,\ldots,g_m;\La(\Om))$, is a collection of elements $w_\psi(\la) \pi(\la) g_i$, where $\la\in\La(\Om)$, $i=1,\ldots, M$ and $w_\psi$ as defined in \eqref{def:w_psi}. The system $\G_\psi(g_1,\ldots,g_m;\La(\Om))$ is a frame for $\Lt(\R^d)$ if there exist constants $A_g,B_g > 0$ such that
\begin{equation}\label{eq:weighted Gabor}
A_g \norm{f}_2^2 \leq \sum_{i=1}^M \sum_{\la\in\La(\Om)} w_\psi(\la)^2 \, \abs{\inner{f}{\pi(\la) g_i}}^2 \leq B_g \norm{f}_2^2
\end{equation}
holds for all $f\in\Lt(\R^d)$. $\G_\psi(g_1,\ldots,g_m;\La(\Om))$ is a {\it normalized weighted tight Gabor frame} if $A_g=B_g=1$. As in the case of non-weighted Gabor systems, $\G_\psi(g_1,\ldots,g_m;\La(\Om))$ is a Bessel sequence (only the right hand side of \eqref{eq:weighted Gabor} holds) if all $g_i\in M^1(\R^d)$, $i=1,\ldots,M$. %The Gabor frame operator associated to $\G_\psi(g_1,\ldots,g_m;\La(\Om))$ takes the form
%\begin{equation*}
%S_g^{\La(\Om),\psi} = \sum_{i=1}^M \sum_{\la\in\La(\Om)} w_\psi(\la)^2 \, \inner{f}{\pi(\la) g_i} \, \pi(\la) g_i\,.
%\end{equation*} 
%If $\G_\psi(h_1,\ldots,h_m;\La(\Om))$ is a weighted Gabor system with  $h_1,\ldots, h_M\in M^(\R^d)$, then we can define a mixed weighted frame operator
%\begin{equation*}
%S_{g,h}^{\La(\Om),\psi} = \sum_{i=1}^M \sum_{\la\in\La(\Om)} w_\psi(\la)^2 \, \inner{f}{\pi(\la) g_i} \, \pi(\la) h_i\,.
%\end{equation*} 
Equipped with this notions, we can now characterize weighted tight Gabor frames.

\begin{proposition}\label{thm:tight frame1}
Let $\La(\Om)$ be a simple model set and $\psi\in C_0^\infty(\Om)$ non-negative function. Then the family $\G_\psi(g_1,\ldots,g_m;\La(\Om))$, with $g_i\in M^1(\R^d)$ for every $i=1,\ldots, M$, is a normalized weighted tight Gabor frame for $\Lt(\R^m)$, that is 
\begin{equation}\label{eq:tight1}
\sum_{i=1}^M \sum_{\la\in\La(\Om)} w_\psi(\la)^2 \, \abs{\inner{f}{\pi(\la) g_i}}^2 = \norm{f}_2^2 \quad \mbox{for all $f\in\Lt(\R^m)$}
\end{equation} 
if and only if
\begin{equation}\label{eq:cond1}
\text{vol}(\Ga)^{-1} \, \widehat{\psi^2}(-p_2^*(\eta))\,\sum_{i=1}^M \innerBig{g_i}{\pi\big(J p_1^*(\eta)\big) g_i} = \delta_{\eta,0}\,,
\end{equation}
for each $\eta\in \Ga^*$, where $\delta$ is the Kronecker delta.
\end{proposition}

\begin{proof}
By  \cite{HW96}, it is sufficient to prove the theorem when $f\in M^1(\R^d)$. Assume that $\G_\psi(g_1,\ldots,g_m;\La(\Om))$ is a normalized weighted tight Gabor frame. Since $g_i\in M^1(\R^d)$ for every $i=1,\ldots, M$, by Proposition~\ref{prop:function N} with $g_i=h_i$ for all $i=1,\ldots,M$ and $f_1=f_2=f$, we can define a function $\mathcal{O}^\psi(z)$ as
\begin{equation}\label{eq:function O}
\mathcal{O}^\psi(z) =  \sum_{i=1}^M \sum_{\la\in\La(\Om)} w_\psi(\la)^2 \,\abs{\inner{\pi(z) f}{\pi(\la) g_i}}^2\,.
\end{equation}
By \eqref{eq:tight1}, this function is constant and equals $\norm{f}_2^2$. Let $\mathcal{E}^\psi(z) = \mathcal{O}^\psi(z) - \norm{f}_2^2$. Then, $\mathcal{E}^\psi(z)$ is almost periodic and $\mathcal{E}^\psi = 0$. By the property of almost periodic function, it implies that the Fourier coefficients of $\mathcal{E}^\psi(z)$,
\begin{equation*}
\widehat{\mathcal{E}^\psi}(\eta) = \left \{ \begin{array}{lc} \widehat{\mathcal{O}^\psi}(\eta) - \norm{f}_2^2, & \eta=0 \\
\widehat{\mathcal{O}^\psi}(\eta), & \eta \neq 0\end{array} \right .
\end{equation*}
are zero. By Proposition~\ref{prop:function N}, we have then
\begin{equation}\label{eq:relation}
\text{vol}(\Ga)^{-1} \, \widehat{\psi^2}(-p_2^*(\eta))\, \innerBig{\pi\big (J p_1^*(\eta)\big) f}{f}\,\sum_{i=1}^M \innerBig{g_i}{\pi\big(J p_1^*(\eta)\big) g_i} = \delta_{\eta,0} \norm{f}_2^2\,,
\end{equation}
with $\eta\in \Ga^*$, for every $f\in M^1(\R^d)$. 

Let $f\in M^1(\R^d)\setminus \{0\}$. Consider first $\eta=0$. Then \eqref{eq:relation} becomes 
\begin{equation*}
\text{vol}(\Ga)^{-1} \, \widehat{\psi^2}(0)\, \inner{f}{f}\,\sum_{i=1}^M \inner{g_i}{g_i} = \delta_{\eta,0} \norm{f}_2^2\,,
\end{equation*}
and \eqref{eq:cond1} follows. Now, let $\eta\neq 0$ be fixed and take $f(x) = e^{-\pi x^2}$, a Gausian. Then \eqref{eq:relation} implies 
\begin{equation*}
\text{vol}(\Ga)^{-1} \, \widehat{\psi^2}(-p_2^*(\eta))\, \innerBig{\pi\big (J p_1^*(\eta)\big) f}{f}\,\sum_{i=1}^M \innerBig{g_i}{\pi\big(J p_1^*(\eta)\big) g_i} = 0\,,
\end{equation*}
and since $\innerBig{\pi\big (J p_1^*(\eta)\big) f}{f} \neq 0$, \eqref{eq:cond1} is satisfied.

Conversely, assume that \eqref{eq:cond1} holds. Since $g_i\in M^1(\R^d)$ for every $i=1,\ldots,M$, by Proposition~\ref{prop:function N} with $f_1=f_2=f$ and $g_i=h_i$ for every $i=1,\ldots,M$, we can define function $\mathcal{O}^\psi(z)$ as in \eqref{eq:function O}. Then, by  Proposition~\ref{prop:function N} and relation \eqref{eq:cond1}, the function $\mathcal{O}^\psi(z)$ is given by the trigonometric series
\begin{equation*}
\sum_{\eta\in \Ga^*} \widehat{\mathcal{O}^\psi}(\eta) e^{-2\pi i z \cdot p_1^*(\eta)} = \mathcal{O}^\psi(z)\,,
\end{equation*}
for every $z\in\R^d$, with
\begin{equation*}
\widehat{\mathcal{O}^\psi}(\eta) = \delta_{\eta,0} \innerBig{\pi\big (J p_1^*(\eta)\big) f}{f}\,.
\end{equation*}
Hence, $\mathcal{O}^\psi(z)$ is constant and $\mathcal{O}^\psi(z) = \norm{f}_2^2$. Evaluating $\mathcal{O}^\psi(z)$ at $z=0$, gives the claim.
\end{proof}

We know state conditions for a weighted Gabor system $\G_\psi(h_1,\ldots,h_m;\La(\Om))$ to be a dual system of $\G_\psi(g_1,\ldots,g_m;\La(\Om))$.

\begin{proposition}\label{thm:dual frame}
Let $\La(\Om)$ be a simple model set and $\psi\in C_0^\infty(\Om)$ non-negative function. Let $g_i\in M^1(\R^d)$ and $h_i\in M^1(\R^d)$, for every $i=1,\ldots, M$. Then $\G_\psi(g_1,\ldots,g_m;\La(\Om))$ and $\G_\psi(h_1,\ldots,h_m;\La(\Om))$ are weighted dual Gabor frames, that is
\begin{equation}\label{eq:dual}
\sum_{i=1}^M \sum_{\la\in\La(\Om)} w_\psi(\la)^2 \, \inner{f}{\pi(\la) g_i}\inner{\pi(\la) h_i}{f} = \norm{f}_2^2 \quad \mbox{for all $f\in\Lt(\R^m)$}\,,
\end{equation} 
if and only if
\begin{equation}\label{eq:cond}
\text{vol}(\Ga)^{-1} \, \widehat{\psi^2}(-p_2^*(\eta)) \,\sum_{i=1}^M \innerBig{h_i}{\pi\big(J p_1^*(\eta)\big) g_i} = \delta_{\eta,0}\,,
\end{equation}
for each $\eta\in \Ga^*$, where $\delta$ is the Kronecker delta.
\end{proposition}

\begin{proof}
The proof is analogous to the proof of Proposition~\ref{thm:tight frame1} with obvious adjustments.
\end{proof}

Based on the last proposition we can also derive density condition for weighted Gabor frames.

\begin{proposition}
Let $\La(\Om)$ be a simple model set and $\psi\in C_0^\infty(\Om)$ a non-negative function such that $\norm{\psi}_2=1$. If the Gabor frame $\G_\psi(g;\La(\Om))$, with $g\in M^1(\La(\Om))$ admits a weighted dual that is also a Gabor system, then $D(\La(\Om))\geq 1$.
\end{proposition}

\begin{proof}
Let $g\in M^1(\R^d)$ and assume that the Gabor frame $\G_\psi(g;\La(\Om))$ admits a dual $\G_\psi(h;\La(\Om))$ with $h\in M^1(\R^d)$. Let $B_g$ be the upper frame bound of $\G_\psi(g;\La(\Om))$, and we can assume without loss of generality that $\norm{h}_2^2 = \abs{\Om}B_g^{-1}$. Then, by polarization, we have the frame decomposition
\begin{equation*}
\inner{f_1}{f_2} = \sum_{\la\in\La(\Om)} \inner{f_1}{\pi(\la) g} w_\psi(\la)^2 \, \inner{\pi(\la) h}{f_2} \quad \mbox{for all $f_1,f_2\in\Lt(\R^m)$,}
\end{equation*} 
If we set $f_1=h$ and $f_2=g$, by the Bessel property of $\G_\psi(g;\La(\Om))$, we obtain
\begin{equation*}
\inner{h}{g} = \sum_{\la\in\La(\Om)}w_\psi(\la)^2 \, \abs{\inner{h}{\pi(\la) g}}^2 \leq B_g \norm{h}_2^2 = \abs{\Om}\,.
\end{equation*}
On the other hand, by Proposition~\ref{thm:dual frame} with $M=1$, $ \text{vol}(\Ga)^{-1} \widehat{\psi^2}(0)\,\inner{h}{g} = 1$, and therefore $D(\La(\Om)) = \text{vol}(\Ga)^{-1}  \abs{\Om}  \geq 1$ because $\widehat{\psi^2}(0) = \norm{\psi}_2^2 = 1$.
\end{proof}

Gabardo, in \cite{Ga09}, gave  a characterization of the weighted irregular Gabor tight frames and dual systems in terms of the distributional symplectic Fourier transform of a positive Borel measure where the windows belong to the Schwartz class. It is possible to derive his results in the setting of model sets, using the characterization just presented.

%%%%%%%%%%%%%%%%%%%%%%%%%%%%%%%%%%%%%%%%%%%%%%%%%%%%%%%%%%%%

\subsection{(Nonweighted) Gabor Systems}

Let $g_i,h_i\in M^1(\R^d)$, $i=1,\ldots,M$, and $\La(\Om)$ a simple model set. At the beginning of Section~\ref{sec:not}, we showed that the frame operator $S_g^{\La(\Om)}$ of $\G(g_1,\ldots,g_M;\La)$ does not commute with the time-frequency shifts taken from $\La(\Om)$. The same holds in particular for any $\La\in X(\La(\Om))$ and a time-frequency shift by $z\in\Rtd$. Let $S_{g,h}^{\La-z}$ denote the mixed frame operator associated to $\G(g_1,\ldots,g_M;\La-z)$ and $\G(h_1,\ldots,h_M;\La-z)$. Then there is a covariance relation relating $S_{g,h}^{\La}$ and $S_{g,h}^{\La-z}$. The  following result was obtained by Kreisel in \cite{K16}. We state it here for the mixed frame operators.

\begin{proposition}\cite{K16}\label{prop:operator shifts}
If $\G(g_1,\ldots,g_M;\La)$ and $\G(h_1,\ldots,h_M;\La)$ ,are Gabor systems for $\La$, and, $\G(g_1,\ldots,g_M;\La-z)$ and $\G(h_1,\ldots,h_M;\La-z)$ are Gabor systems for $\La-z$, then the mixed frame operators $S_{g,h}^\La$ and $S_{g,h}^{\La-z}$ satisfy
\begin{equation*}
S_{g,h}^\La \, \pi(z) = \pi(z)\, S_{g,h}^{\La-z}\,.
\end{equation*}
\end{proposition}

Moreover, the following continuity property holds.
\begin{proposition}\cite{K16}\label{prop:continuity}
Suppose $\La_n\rightarrow \La$ in $X(\La(\Om))$ and the window functions $g_i,h_i$ lie in $M^1(\R^d)$, for each $i=1,\ldots,M$. Then $S_{g,h}^{\La_n} \rightarrow S_{g,h}^{\La}$ in the strong operator topology on $B(M^1(\R^d))$.
\end{proposition} 

Even though the mapping $\La \rightarrow S_g^{\La}$, $\La\in X(\La(\Om))$ is not continuous when $B(M^1(\R^d))$ is given the norm topology, all the frames $\G(g_1,\ldots,g_M;\La)$ have the same optimal frame bounds.
\begin{proposition}\cite{K16}\label{prop:all shifts}
Suppose $\G(g_1,\ldots,g_M;\La)$ is a frame for each $\La \in X(\La(\Om))$ and each $g_i\in M^1(\R^d)$. For any $\La\in X(\La(\Om))$ the optimal upper and lower frame bounds for $\G(g_1,\ldots,g_M;\La)$ are the same as those for $\G(g_1,\ldots,g_M;\La(\Om))$. As a result, $\norm{S_g^\La}_{M^1} = \norm{S_g^{\La(\Om)}}_{M^1}$ and $\norm{(S_g^\La)^{-1}}_{M^1} = \norm{(S_g^{\La(\Om)})^{-1}}_{M^1}$, where $\norm{\cdot}_{M^1}$ denotes the operator norm on $B(M^1(\R^d))$.
\end{proposition}
As  a result of the continuity property, we have the following Corollary.
\begin{corollary}\cite{K16}
Suppose $g_1,\ldots,g_M\in M^1(\R^d)$ and $\G(g_1,\ldots,g_M;\La(\Om))$ is an $M^1$-frame. Then for any $\La \in X(\La(\Om))$, $\G(g_1,\ldots,g_M;\La)$ is also an $M^1$-frame.
\end{corollary}

%Applying results from Section~\ref{sec:bracket} we will develop Janssen representation of the Gabor frame operator for $\La(\Om)$ and Wexler-Raz biorthogonality relations, and as a result, specify the conditions on $g_i\in M^1(\R^d)$, $i=1,\ldots,M$, making $\G(g_1,\ldots,g_M;\La(\Om))$ a normalized tight frame. 

Now, let $f_1,f_2 \in M^1(\R^d)$ be fixed and let $g_i,h_i\in M^1(\R^d)$ for $i=1,\ldots,M$. We define a function $\widetilde{\mathcal{N}} : X(\La(\Om)) \rightarrow \C$ through the mixed frame operator, as
\begin{equation*}
\widetilde{\mathcal{N}}(\La) = \inner{S_{g,h}^{\La}f_1}{f_2}\,.
\end{equation*}
Since, by Proposition~\ref{prop:continuity}, $S_{g,h}^{\La}$ is continuous, in the strong operator topology, over $X(\La(\Om))$, the function $\widetilde{\mathcal{N}}$ is continuous. As was presented in Section~\ref{sec:ms}, we can define from $\widetilde{\mathcal{N}}$ a function $\mathcal{N} :\R^m\rightarrow \C$, by
\begin{equation}\label{eq:main function}
\mathcal{N}(z) = \widetilde{\mathcal{N}}(\La(\Om)-z)\,,
\end{equation}
and since $\widetilde{\mathcal{N}}$ is continuous, $\mathcal{N}$ is local with respect to $\La(\Om)$. As was shown in Section~\ref{sec:ms}, it has a Fourier expansion
\begin{equation*}
\mathcal{N}(z) = \sum_{\eta\in\Ga^*} \widehat{\mathfrak{N}}(\eta) e^{-2\pi i p_1^*(\eta) \cdot z}
\end{equation*}
where
\begin{equation}\label{eq:fourier coeff}
\widehat{\mathfrak{N}}(\eta)  = \lim_{R\rightarrow \infty} \frac{1}{R^{2d}} \int_{B(0,R)} \mathcal{N}(z) e^{2\pi i z \cdot p_1^*(\eta)}\, dz\,.
\end{equation}
Applying the tools developed in Section~\ref{sec:bracket}, we will be able to compute the Fourier coefficients $\widehat{\mathfrak{N}}(\eta)$ of $\mathcal{N}$. 

Before we proceed further we introduce a sequence of auxiliary functions that will be crucial in proving our results. The following Lemma is a particular case of  Prop~$3.6$ in \cite{BM00}. 
\begin{lemma}\label{lemma:psi}
Let $0<\epsilon<1$, $\Om$ be a compact subset of $\R$, $\widetilde{\Om} = (1-\epsilon)\Om$ and $f_s =  \frac{\mathds{1}_{\epsilon^{s}\widetilde{\Om}}}{\abs{\epsilon^{s}\widetilde{\Om}}}$ for $s\in \N$. Then,
\begin{itemize}
\item[(i)] the infinite convolution product $\Convtext_{s=0}^{\infty} \, f_s$ converges in $L^1(\R)$ and defines a non-negative smooth function $\psi = \Convtext_{s=0}^{\infty} \, f_s$ compactly supported on $\Om$, with $\norm{\psi}_1 = 1$;
\item[(ii)] the Fourier transform of $\psi$ is the smooth function $\widehat{\psi} = \prod_{s=0}^{\infty} \,\widehat{f_s} $, with uniform convergence of the product, where $\widehat{f_s}(t) = \text{sinc}(t\,\abs{\epsilon^{s} \widetilde{\Om}})$ and $\text{sinc}(t) = \frac{\sin(\pi t)}{\pi t}$.
\end{itemize}
\end{lemma}

Now, for $0<\epsilon<1$ and $\Om$ a compact subset of $\R$, we define a sequence of compact sets $\Om_n = (1-\epsilon^n)\Om$. The sets $\Om_n$ are increasing and $\overline{\bigcup_n \Om_n} = \Om$. Let $\psi_n$ be an infinite convolution product
\begin{equation}\label{eq:function psi_n}
\psi_n= \Conv_{s=0}^{\infty} \, \frac{\mathds{1}_{\epsilon^{ns}\Om_n}}{\abs{\epsilon^{ns}\Om_n}} = 
\frac{\mathds{1}_{\Om_n}}{\abs{\Om_n}} \ast \left ( \Conv_{s=1}^{\infty} \, \frac{\mathds{1}_{\epsilon^{ns}\Om_n}}{\abs{\epsilon^{ns}\Om_n}}\right)\,.
\end{equation}
Then, by Lemma~\ref{lemma:psi}, each $\psi_n$ is well defined and forms a  sequence of $C_0^\infty(\Om)$ non-negative functions, with Fourier transform of $\psi_n$ being
\begin{equation}\label{eq:function Fourier psi_n}
\widehat{\psi_n}(t) = \prod_{s=0}^{\infty} \,\text{sinc}(t\,\abs{\epsilon^{ns} \Om_n}) = 
\text{sinc}(t\,\abs{\Om_n})  \cdot \prod_{s=1}^{\infty} \, \text{sinc}(t\,\abs{\epsilon^{ns} \Om_n}) \,,.
\end{equation}

\begin{lemma}\label{lemma:psi_n}
With the above notation, 
\begin{itemize}
\item[(i)] the sequence $\{\psi_n\}_{n=1}^{\infty}$ converges pointwise to $\frac{\mathds{1}_\Om}{\abs{\Om}}$ on $\Om\setminus \partial \Om$, where $\partial \Om$ is the boundary of $\Om$;
\item[(ii)] the sequence of Fourier transforms, $\{\widehat{\psi_n}\}_{n=1}^{\infty}$, converges uniformly.
\end{itemize}
\end{lemma}

\begin{proof}
Let $t_0\in \Om\setminus \partial \Om$ and $\delta>0$. By the properties of the sets $\Om_n$ we have: $\Om_m \subset \Om_n \subset \Om$ and $\epsilon^{ns} \Om_n \subset \epsilon^{ms} \Om_m$ for $n\geq m$ and $\epsilon^{ns} \Om_n \subset \epsilon^{n(s+1)} \Om_n$ for all $n$ and $s$. Then, there exists $N\geq 0$, such that for all $n\geq N$, $\abs{\Om \setminus \Om_n} < \delta \abs{\Om}^2$ and $t_0\in\Om_n$. That means, for $n\geq N$
\begin{equation*}
\left [ \left (\sum_{s=1}^{\infty} \text{supp }\frac{\mathds{1}_{\epsilon^{ns}\Om_n}}{\abs{\epsilon^{ns}\Om_n}} \right )- t_0\right ] \cap \Om_n = \left [\left (\sum_{s=1}^\infty \epsilon^{ns}\Om_n - t_0 \right )\right ] \cap \Om_n \neq  \emptyset\,.
\end{equation*}
Then, for all $n\geq N$,
\begin{align*}
\left |\psi_n(t_0) - \frac{\mathds{1}_{\Om}(t_0)}{\abs{\Om}} \right | &= \left | \left [\frac{\mathds{1}_{\Om_n}}{\abs{\Om_n}} \ast \left ( \Conv_{s=1}^{\infty} \, \frac{\mathds{1}_{\epsilon^{ns}\Om_n}}{\abs{\epsilon^{ns}\Om_n}}\right)\right ](t_0) - \frac{1}{\abs{\Om}} \right | \\
&= \left | \frac{1}{\abs{\Om_n}} \int_{\Om_n}  \left ( \Conv_{s=1}^{\infty} \, \frac{\mathds{1}_{\epsilon^{ns}\Om_n}}{\abs{\epsilon^{ns}\Om_n}}\right)(t_0-x)\, dx -  \frac{1}{\abs{\Om}} \right | \\
&\leq \left | \frac{1}{\abs{\Om_n}}\left \| \Conv_{s=1}^{\infty} \, \frac{\mathds{1}_{\epsilon^{ns}\Om_n}}{\abs{\epsilon^{ns}\Om_n}}\right \|_1 -  \frac{1}{\abs{\Om}} \right | \\
&= \left | \frac{1}{\abs{\Om_n}} -  \frac{1}{\abs{\Om}} \right | = \frac{\abs{\Om \setminus \Om_n}}{\abs{\Om_n}\abs{\Om}} \\
&< \frac{\delta}{(1-\epsilon^N)}\,,
\end{align*}
and claim $(i)$ follows.

For $(ii)$ it suffices to show that the sequence of Fourier transforms $\widehat{\psi_n}$ is uniformly Cauchy. By $(i)$ and Lemma~\ref{lemma:psi}, $\{\psi_n\}_{n=1}^{\infty}$ is a sequence of $L^1$ functions converging pointwise almost everywhere to $\frac{\mathds{1}_{\Om}}{\abs{\Om}} $.  Then by the $L^1$ Dominated Convergence Theorem, $\{\psi_n\}_{n=1}^{\infty}$ converges to $\frac{\mathds{1}_{\Om}}{\abs{\Om}} $ in $L^1$, and it follows that $\{\psi_n\}_{n=1}^{\infty}$  is an $L^1$ Cauchy sequence. Meaning, for $\delta>0$ there exists $N>0$ such that $\norm{\psi_n-\psi_m}_1 < \delta$ for all $n,m\geq N$.  Let $n,m\geq N$, then
\begin{equation*} 
\norm{\widehat{\psi_m} - \widehat{\psi_n}}_\infty \leq \norm{\psi_m - \psi_n}_1 < \delta\,,
\end{equation*}
and $\{\widehat{\psi_n}\}_{n=1}^{\infty}$ is uniformly Cauchy. By the completeness of $L^{\infty}(\R)$, it converges uniformly.
\end{proof}

Let, $\Psi$ denote the uniform limit of the sequence defined in \eqref{eq:function Fourier psi_n}. To make things more convenient later, we normalize $\Psi$, and define a new function
\begin{equation}\label{eq:function phi} 
\Phi = \abs{\Om} \cdot (\conv{\Psi}{\Psi})\,. 
\end{equation}
Note that $\Phi(0)=1$. 
 
The following observation will be the main ingredient in our approach. It is analogous to the results for lattices developed in Section~\ref{sec:gabor frames lattice}. With the above notation we have

\begin{proposition}\label{prop:limit function N}
Let $\La(\Om)$ be a simple generic model set. Assume that $g_i,h_i\in M^1(\R^d)$ for every $i=1,\ldots,M$. Then, for every $f_1,f_2\in M^1(\R^d)$, the function
\begin{equation*}
\mathcal{N}(z) = \sum_{i=1}^M \sum_{\la\in\La(\Om)} \innerBig{\pi(z) f_1}{\pi(\la) g_i} \innerBig{\pi(\la) h_i}{\pi(z) f_2}
\end{equation*}
is continuous and coincides pointwise with its Fourier series $\sum_{\eta\in\Ga^*} \widehat{\mathcal{N}}(\eta) e^{-2\pi i p_1^*(\eta) \cdot z}$, with
\begin{equation*}
\widehat{\mathcal{N}}(\eta) = D(\La(\Om)) \, \Phi(-p_2^*(\eta)) \innerBig{\pi\big (J p_1^*(\eta)\big) f_1}{f_2}\,\sum_{i=1}^M \innerBig{h_i}{\pi\big(J p_1^*(\eta)\big) g_i} \,.
\end{equation*}
where $\eta\in \Ga^\ast$, $\Phi$ defined in \eqref{eq:function phi}.
\end{proposition}
We can approximate function $\mathcal{N}$ with the desired accuracy by an almost periodic function whose spectrum lies in a 'dual' model set. Let $\epsilon>0$ and $C=\max{\norm{\V_{g_i} h_i}_{\mathbf{W}(L^\infty,\lt)}}$. The function $\Phi$ decays rapidly outside its essential support, hence, for the essential support, we can choose a compact interval $\widetilde{\Om}_\epsilon$, depending on the windows $g_i,h_i$, such that 
\begin{equation}\label{eq:estimate}
%%\sum_{\eta\in\Ga^*} \left (\Big(\absbig{\Phi}^* \mathds{1}_{\widetilde{\Om}^c}\Big ) \otimes \Big (\sum_{i=1}^M \absbig{\V_{g_i} h_i}^2\Big ) \right )(\eta)  < \epsilon\,,
%\normBig{\Big(\absbig{\Phi}^* \mathds{1}_{\widetilde{\Om}^c}\Big ) \otimes \Big (\sum_{i=1}^M \absbig{\V_{g_i} h_i}^2\Big )}_{\mathbf{W}} < \frac{\epsilon}{D(\La(\Om))\,\text{rel}(\Ga^*)^2}
\normBig{\Phi \mathds{1}_{\widetilde{\Om}\epsilon^c}}_{\mathbf{W}(\R)} < \frac{\epsilon}{D(\La(\Om))\, \text{rel}(\Ga^*)\, C\, M}
\end{equation} 
where $\absbig{\Phi}^*(x) = \absbig{\Phi}(-x)$ is the involution and $\widetilde{\Om}_\epsilon^c$ a complement of $\widetilde{\Om}_\epsilon$. Let us define an $\epsilon-$dual model set $\La^*(\widetilde{\Om}_\epsilon)$ originating from $\Ga^*$ and $\widetilde{\Om}_\epsilon$ by
\begin{equation}\label{eq:epsilon dula model set}
\La^*(\widetilde{\Om}_\epsilon) = \left \{\be = p_1^*(\eta)\,:\, \eta\in\Ga^*,\, p_2^*(\eta) \in \widetilde{\Om}_\epsilon \right \}
\end{equation}
We note here, that the concept of an $\epsilon$-dual model set defined here differs from the original $\epsilon$-dual model set definition by Meyer.  Then, the function $\mathcal{N}_{\epsilon}$, given by the series 
\begin{equation*}
\mathcal{N}_\epsilon (z) := \sum_{\be\in\La^*(\widetilde{\Om}_\epsilon)} \widehat{\mathcal{N}}(\be) e^{-2\pi i \be \cdot z}\,,
\end{equation*}
defines an almost periodic function with spectrum in $\La^*(\widetilde{\Om}_\epsilon)$. Moreover, by Cauchy-Schwarz inequality, we have 
\begin{align*}
\normbig{&\mathcal{N} - \mathcal{N}_\epsilon}_\infty \leq  D(\La(\Om)) \sum_{i=1}^M \sum_{\eta\in\Ga^*;\, p_2^*(\eta)\notin \widetilde{\Om}_\epsilon} \absBig{\Phi(-p_2^*(\eta)) \innerBig{\pi\big (J p_1^*(\eta)\big) f_1}{f_2}\,\innerBig{h_i}{\pi\big(J p_1^*(\eta)\big) g_i}}\\
&\leq D(\La(\Om))\sum_{i=1}^M \left [ \sum_{\eta\in\Ga^*} \Big (\Big(\absbig{\Phi}^*\mathds{1}_{\widetilde{\Om}_\epsilon^c} \Big)\otimes \U\absbig{\V_{g_i} h_i}^2 \Big)(\eta)\right]^{1/2}\left [ \sum_{\eta\in\Ga^*} \Big( \Big(\absbig{\Phi}^*\mathds{1}_{\widetilde{\Om}_\epsilon^c}\Big ) \otimes \U\absbig{\V_{f_1} f_2}^2 \Big)(\eta)\right]^{1/2}\\
&\leq D(\La(\Om)) \, \text{rel}(\Ga^*)\,\sum_{i=1}^M \normBig{\Big(\absbig{\Phi}^* \mathds{1}_{\widetilde{\Om}_\epsilon^c}\Big ) \otimes \U \absbig{\V_{g_i} h_i}^2}_{\mathbf{W}(\R\times \Rtd)}^{1/2} \normBig{\Big(\absbig{\Phi}^* \mathds{1}_{\widetilde{\Om}_\epsilon^c}\Big ) \otimes \U \absbig{\V_{f_1} f_2}^2}_{\mathbf{W}(\R\times \Rtd)}^{1/2}\,.
\end{align*}
By the property of tensor product and \eqref{eq:estimate}, we obtain
\begin{align*}
\normbig{\mathcal{N} - \mathcal{N}_\epsilon}_\infty &\leq D(\La(\Om))\, \text{rel}(\Ga^*)\, \normBig{\Phi \mathds{1}_{\widetilde{\Om}_\epsilon^c}}_{\mathbf{W}(\R)} \sum_{i=1}^M  \normbig{\abs{V_{g_i} h_i}^2}_{\mathbf{W}(\Rtd)}\, \normbig{\abs{V_{f_1} f_2}^2}_{\mathbf{W}(\Rtd)}\\
&\leq D(\La(\Om))\, \text{rel}(\Ga^*)\, \normBig{\Phi \mathds{1}_{\widetilde{\Om}_\epsilon^c}}_{\mathbf{W}(\R)} \sum_{i=1}^M  \normbig{V_{g_i} h_i}_{\mathbf{W}(L^\infty,\lt)(\Rtd)}\, \normbig{V_{f_1} f_2}_{\mathbf{W}(L^\infty,\lt)(\Rtd)}\\
&< \epsilon \, \normbig{V_{f_1} f_2}_{\mathbf{W}(L^\infty,\lt)(\Rtd)}\,.
\end{align*}

\begin{proof}[Proof of Proposition~\ref{prop:limit function N}]
Let $f_1,f_2\in M^1(\R^d)$. Let $\psi_n$ be a  sequence of $C_0^\infty(\Om)$ non-negative functions defined in \eqref{eq:function psi_n}. By Proposition~\ref{prop:function N}, for each $n\in\N$, the functions
\begin{equation*}
\mathcal{N}^{\psi_n}(z) = \sum_{i=1}^M \sum_{\la\in\La(\Om)} w_{\psi_n}(\la)^2\, \innerBig{\pi(z) f_1}{\pi(\la) g_i} \innerBig{\pi(\la) h_i}{\pi(z) f_2}
\end{equation*}
are well defined almost periodic functions that are pointwise equal to their Fourier series $\sum_{\eta\in\Ga^*} \widehat{\mathcal{N}^{\psi_n}}(\eta) e^{-2\pi i p_1^*(\eta) \cdot z}$, where $\widehat{\mathcal{N}^{\psi_n}}(\eta)$ are given by
\begin{equation*}
\widehat{\mathcal{N}^{\psi_n}}(\eta) =  \text{vol}(\Ga)^{-1} \, \widehat{\psi_n^2}(-p_2^*(\eta))\, \innerBig{\pi\big (J p_1^*(\eta)\big) f_1}{f_2}\,\sum_{i=1}^M \innerBig{h_i}{\pi\big(J p_1^*(\eta)\big) g_i}\,.
\end{equation*}

On the other hand, the series $\sum_{\eta\in\Ga^*} \widehat{\mathcal{N}}(\eta) e^{-2\pi i p_1^*(\eta) \cdot z}$ converges absolutely (by a similar argument as in the proof of Proposition~\ref{prop:function N}) and gives rise to a uniformly continuous function. By the uniform convergence of $\widehat{\psi_n^2}$ to $\conv{\Psi}{\Psi}$, it can be easily verified that $\mathcal{N}^{\psi_n}$ converges uniformly to $\abs{\Om}^{-2}\,\sum_{\eta\in\Ga^*} \widehat{\mathcal{N}}(\eta) e^{-2\pi i p_1^*(\eta) \cdot z}$, since $D(\La(\Om)) = \text{vol}(\Ga)^{-1} \abs{\Om}$ and $\Phi= \abs{\Om} \cdot (\conv{\Psi}{\Psi})$.

Now, since $\La(\Om)$ is genereic, that is the boundary $\partial \Om$ of $\Om$ has no common points with $p_2(\Ga)$,  and $\psi_n$ converges pointwise to $\frac{\mathds{1}_\Om}{\abs{\Om}}$ on $\Om\setminus \partial \Om$, by Lemma~\ref{lemma:psi}, we show that $\mathcal{N}(z)$ is a pointwise limit of $\abs{\Om}^2 \,\mathcal{N}^{\psi_n}(z)$. Indeed, by the Lebesgue Dominated Convergence Theorem, we can move the limit inside the sum, and for every $z\in\Rtd$, we have
\begin{align*}
\lim_{n\rightarrow \infty} \mathcal{N}^{\psi_n}(z) &=  \sum_{i=1}^M \sum_{\la\in\La(\Om)} \lim_{n\rightarrow \infty} w_{\psi_n}^2(\la) \,  \innerBig{\pi(z) f_1}{\pi(\la) g_i} \innerBig{\pi(\la) h_i}{\pi(z) f_2}\\
& = \sum_{i=1}^M \sum_{\la\in\La(\Om)} \lim_{n\rightarrow \infty} \psi_n^2(p_2(\ga)) \,  \innerBig{\pi(z) f_1}{\pi(\la) g_i} \innerBig{\pi(\la) h_i}{\pi(z) f_2} \quad (\la=p_1(\ga),\, \ga\in\Ga)\\
&= \sum_{i=1}^M \sum_{\la\in\La(\Om)} \abs{\Om}^{-2}\big [ \mathds{1}_{\Om}(p_2(\ga))\big ]^2 \, \innerBig{\pi(z) f_1}{\pi(\la) g_i} \innerBig{\pi(\la) h_i}{\pi(z) f_2}\\
& =  \abs{\Om}^{-2} \mathcal{N}(z)\,.
\end{align*}
By the uniqueness of the the limits, we must have $\mathcal{N}(z) = \sum_{\eta\in\Ga^*} \widehat{\mathcal{N}}(\eta) e^{-2\pi i p_1^*(\eta) \cdot z}$, and by the uniqueness of the Fourier series, \eqref{eq:series} is the Fourier series of $\mathcal{N}(z)$.
\end{proof}

For $g_i,h_i\in S_(\R^d)$ with $i=1,\ldots,M$, the function $\mathcal{N}(z)$ of Proposition~\ref{prop:limit function N} coincides with the function $\mathcal{N}(z)$ defined in \eqref{eq:main function}. Indeed, using Proposition~\ref{prop:operator shifts}, we can write $\mathcal{N}(z)$ from \eqref{eq:main function} explicitly as
\begin{align*}
\mathcal{N}(z) &=  \widetilde{\mathcal{N}}(\La(\Om)-z) = \innerBig{S_{g,h}^{\La(\Om)-z}f_1}{f_2} = \innerBig{S_{g,h}^{\La(\Om)}\pi(z)f_1}{\pi(z) f_2}\\ 
&= \sum_{i=1}^M \sum_{\la\in\La(\Om)} \innerBig{\pi(z) f_1}{\pi(\la) g_i} \innerBig{\pi(\la) h_i}{\pi(z) f_2}\,.
\end{align*}
By the uniqueness of the Fourier coefficients, $\widehat{\mathfrak{N}}(\eta)$ in \eqref{eq:fourier coeff} equal $\widehat{\mathcal{N}}(\eta)$ from Proposition~\ref{prop:limit function N}, for all $\eta\in\Ga^*$.

We are now in the position to state the main results.
\begin{theorem}\label{thm:main thm}
Let $\La(\Om)$ be a simple generic model set. Then for $g_i,h_i\in M^1(\R^d)$, $i=1,\ldots,M$ the following hold.
\begin{itemize}
\item[i)] Fundamental Identity of Gabor Analysis for Model Sets:
\begin{align}\label{eq:FIGAMS}
\sum_{i=1}^M \sum_{\la\in\La(\Om)} &\inner{f_1}{\pi(\la)g_i}\inner{\pi(\la)h_i}{f_2} \nonumber\\
&=  D(\La(\Om)) \sum_{\eta\in\Ga^*} \Phi(-p_2^*(\eta)) \innerBig{\pi\big (J p_1^*(\eta)\big) f_1}{f_2}\,\sum_{i=1}^M \innerBig{h_i}{\pi\big(J p_1^*(\eta)\big) g_i}
\end{align}
for all $f_1,f_2\in M^1(\R^d)$.
\item[ii)] Janssen Representation:
\begin{equation}\label{eq:Janssen MS}
S_{g,h}^{\La(\Om)} =D(\La(\Om)) \sum_{i=1}^M \sum_{\eta\in\Ga^*} \Phi(-p_2^*(\eta)) \, \innerBig{h_i}{\pi\big(J p_1^*(\eta)\big) g_i}\, \pi\big (J p_1^*(\eta)\big)\,,
\end{equation}\
where the series converges unconditionally in the strong operator topology.
\item[iii)] Wexler-Raz Biorthogonality Relations:
\begin{equation}\label{eq:WR MS}
S_{g,h}^{\La(\Om)} = I \mbox{  on $M^1(\R^d)$} \quad   \Longleftrightarrow \quad D(\La(\Om)) \sum_{i=1}^M \Phi(-p_2^*(\eta)) \, \innerBig{h_i}{\pi\big(J p_1^*(\eta)\big) g_i} = \delta_{\eta,0}
\end{equation}
for $\eta\in\Ga^*$.
\end{itemize}
\end{theorem}

The relation \eqref{eq:FIGAMS}, as well as \eqref{eq:Janssen MS}, can be written using dual model sets, and giving a better understanding of the above relations to the ones for regular lattices. Let $\widetilde{\Om}$ be a compact subset of $\R$, equal closure of its interior and with measure of the boundary equal to zero. We define a sequence of dual model sets as
\begin{equation*}
\La^*_m(\widetilde{\Om}) = \left \{\beta=p_1^*(\eta)\,:\,\eta\in\Ga^*,\, p_2^*(\eta) \in m\widetilde{\Om}\setminus (m-1)\widetilde{\Om} \right \}\,.
\end{equation*}
Then, the right hand side of  \eqref{eq:FIGAMS} equals
\begin{align*}
D(\La(\Om)) &\sum_{\eta\in\Ga^*} \Phi(-p_2^*(\eta)) \innerBig{\pi\big (J p_1^*(\eta)\big) f_1}{f_2}\,\sum_{i=1}^M \innerBig{h_i}{\pi\big(J p_1^*(\eta)\big) g_i}\\
&= D(\La(\Om)) \sum_{m=1}^\infty \sum_{\beta \in \La^*_m(\widetilde{\Om})} w_{\Phi}(-\beta) \innerBig{\pi\big (J \beta \big) f_1}{f_2}\,\sum_{i=1}^M \innerBig{h_i}{\pi\big(J \beta \big) g_i}\,,
\end{align*}
where $w_{\Phi}$ is defined as in \eqref{def:w_psi} for the function $\Phi$. Since $\Phi$ is well concentrated, with rapid decay outside its essential support, the sum over $m$ has only a finite number of relevant terms, and we can approximate \eqref{eq:FIGAMS} as
\begin{align*}
\sum_{i=1}^M \sum_{\la\in\La(\Om)} &\inner{f_1}{\pi(\la)g_i}\inner{\pi(\la)h_i}{f_2} \\
&\approx  D(\La(\Om)) \sum_{m=1}^M \sum_{\be\in \La^*_m(\widetilde{\Om})} w_{\Phi}(-\be) \innerBig{\pi\big (J \be\big) f_1}{f_2}\,\sum_{i=1}^M \innerBig{h_i}{\pi\big(J \be \big) g_i}\,.
\end{align*}

Approximation depends on $M$ and $\widetilde{\Om}$. We can choose $\widetilde{\Om}$ to be an  $\epsilon-$essential support of $\Phi$ in the sense of \eqref{eq:estimate}, that is $\widetilde{\Om} = \widetilde{\Om}_\epsilon$ and we have $\La^*_1(\widetilde{\Om}) = \La^*(\widetilde{\Om}_\epsilon)$ as in \eqref{eq:epsilon dula model set}.  Then we obtain a good approximation already for $M=1$
\begin{itemize}
\item[i)] Fundamental Identity of Gabor Analysis for Model Sets becomes
\begin{align*}
\sum_{i=1}^M \sum_{\la\in\La(\Om)} \inner{f_1}{\pi(\la)g_i} &\inner{\pi(\la)h_i}{f_2} \nonumber\\
&\approx  D(\La(\Om)) \sum_{\be\in \La^*(\widetilde{\Om}_\epsilon)} w_{\Phi}(-\be) \innerBig{\pi(J \be) f_1}{f_2}\,\sum_{i=1}^M \innerBig{h_i}{\pi(J \be) g_i}\,.
\end{align*}
\item[ii)] Janssen Representation gives us an approximation of the frame operator in the form of
\begin{equation*}
S_{g,h}^{\La(\Om)} \approx D(\La(\Om)) \sum_{i=1}^M  \sum_{\be\in \La^*(\widetilde{\Om}_\epsilon)} w_{\Phi}(-\be) \, \innerBig{h_i}{\pi(J \be) g_i}\, \pi(J \be)\,.
\end{equation*}
\item[iii)] Wexler-Raz Biorthogonality Relations provide an approximation to the identity operator:
\begin{equation*}
S_{g,h}^{\La(\Om)} \approx I \mbox{  on $M^1(\R^d)$} \quad   \Longleftrightarrow \quad D(\La(\Om)) \sum_{i=1}^M w_{\Phi}(-\be) \, \innerBig{h_i}{\pi(J \be) g_i} = \delta_{\be,0}
\end{equation*}
for $\be\in\La^*(\widetilde{\Om}_\epsilon)$.
\end{itemize}
These relations resemble the relations of Gabor systems for lattices, where there is a connection between a lattice and its dual (or symplectic dual). Here $\La(\Om)$ takes the place of a lattice, and an $\epsilon-$dual model set $\La^*(\widetilde{\Om}_\epsilon)$ takes the place of a dual lattice, and depends on the desired accuracy of the approximation and window functions $g_i,h_i$.

\begin{proof}
The Fundamental Identity of Gabor Analysis and Janssen representation follow directly from Proposition~\ref{prop:limit function N}. Let $f_1,f_2\in M^1(\R^d)$, then the left hand side of \eqref{eq:FIGAMS} equals the function $\mathcal{N}$ of Proposition~\ref{prop:limit function N} evaluated at $z=0$. Since $\mathcal{N}$ equals its Fourier series expansion, we have 
\begin{equation*}
\mathcal{N}(0) =   D(\La(\Om)) \sum_{\eta\in\Ga^*} \Phi(-p_2^*(\eta)) \innerBig{\pi\big (J p_1^*(\eta)\big) f_1}{f_2}\,\sum_{i=1}^M \innerBig{h_i}{\pi\big(J p_1^*(\eta)\big) g_i}\,.
\end{equation*}
which gives \eqref{eq:FIGAMS}. 

As for Janssen representation, we observe that $\mathcal{N}(z) = \inner{S_{g,h}^{\La(\Om)} \pi(z)f_1}{\pi(z) f_2}$ for fixed $f_1,f_2\in M^1(\R^d)$. Evaluating $\mathcal{N}$ at $z=0$ and using the Fourier series representation of $\mathcal{N}$, we obtain
\begin{align*}
\inner{S_{g,h}^{\La(\Om)} f_1}{f_2} &= D(\La(\Om)) \sum_{\eta\in\Ga^*} \Phi(-p_2^*(\eta)) \innerBig{\pi\big (J p_1^*(\eta)\big) f_1}{f_2}\,\sum_{i=1}^M \innerBig{h_i}{\pi\big(J p_1^*(\eta)\big) g_i}\\
&= \innerBig{D(\La(\Om)) \sum_{\eta\in\Ga^*} \Phi(-p_2^*(\eta)) \, \sum_{i=1}^M \innerBig{h_i}{\pi\big(J p_1^*(\eta)\big) g_i} \pi\big (J p_1^*(\eta)\big) f_1}{f_2}\,,
\end{align*}
which is the Janssen representation of the frame operator $S_{g,h}^{\La(\Om)}$.

The implication $\Longleftarrow$ of $iii)$ follows trivially from the Janssen representation of $S_{g,h}^{\La(\Om)}$. For the converse, assume that $S_{g,h}^{\La(\Om)} = I$. Let $f_1,f_2 \in M^1(\R^d)$, then $\mathcal{N}$ of Proposition~\ref{prop:function N lattice} is a constant function. Indeed, 
\begin{equation*}
\inner{f_1}{f_2} = \inner{\pi(z)f_1}{\pi(z)f_2} = \inner{S_{g,h}^{\La(\Om)} \pi(z)f_1}{\pi(z)f_2} = \mathcal{N}(z)
\end{equation*} 
for every $z\in\Rtd$. Since $\mathcal{O} = \mathcal{N} - \inner{f_1}{f_2}$ is a zero function, all its Fourier coefficients
\begin{equation*}
\widehat{\mathcal{O}}(\eta) = \left \{ \begin{array}{lc} \widehat{\mathcal{N}}(\eta) - \inner{f_1}{f_2}, & \eta=0 \\
\widehat{\mathcal{N}}(\eta), & \eta \neq 0\end{array} \right .
\end{equation*}
are zero. By Proposition~\ref{prop:limit function N}, we have then
\begin{equation}\label{eq:rel2}
D(\La(\Om)) \sum_{i=1}^M \Phi(-p_2^*(\eta)) \, \innerBig{h_i}{\pi\big(J p_1^*(\eta)\big) g_i}\, \innerBig{\pi\big(J p_1^*(\eta)\big) f_1}{f_2} = \delta_{\eta,0} \inner{f_1}{f_2}\,.
\end{equation}
Fix $J p_1^*(\eta)\in \Rtd$ and let $f$ be a Gausian, that is $f(x) = e^{-\pi x^2}$. Then $\innerBig{\pi\big(J p_1^*(\eta)\big) f_1}{f_2} = \innerBig{\pi\big(J p_1^*(\eta)\big) f}{f} \neq 0$ and the right hand side of \eqref{eq:WR MS} holds.
\end{proof}

As a consequence of the Wexler-Raz biorthogonality relations we obtain a density result for Gabor systems for model sets.

\begin{proposition}
Let $\La(\Om)$ be a simple generic model set. If the Gabor frame $\G(g;\La(\Om))$, with $g\in M^1(\La(\Om))$ admits a dual that is also a Gabor system, then $D(\La(\Om))\geq 1$.
\end{proposition}

\begin{proof}
Let $g\in M^1(\R^d)$ and assume that the Gabor frame $\G(g;\La(\Om))$ admits a dual $\G(h;\La(\Om))$ with $h\in M^1(\R^d)$. Let $B_g$ be the upper frame bound of $\G(g;\La(\Om))$, and we can assume without loss of generality that $\norm{h}_2^2 = B_g^{-1}$. Then we have the frame decomposition
\begin{equation*}
\inner{f_1}{f_2} = \sum_{\la\in\La(\Om)} \inner{f_1}{\pi(\la) g} \inner{\pi(\la) h}{f_2} \quad \mbox{for all $f_1,f_2\in\Lt(\R^m)$,}
\end{equation*} 
If we set $f_1=h$ and $f_2=g$, by the Bessel property of $\G(g;\La(\Om))$, we obtain
\begin{equation*}
\inner{h}{g} = \sum_{\la\in\La(\Om)} \abs{\inner{h}{\pi(\la) g}}^2 \leq B_g \norm{h}_2^2 = 1\,.
\end{equation*}
On the other hand, by Theorem~\ref{thm:main thm}~$iii)$ with $L=1$, $D(\La(\Om)) \inner{h}{g} = 1$, and therefore $D(\La(\Om)) \geq 1$.
\end{proof}

%%%%%%%%%%%%%%%%%%%%%%%%%%%%%%%%%%%%%%%%%%%%%%%%%%%%%%%%%%%%

%\section{Conclusions}\label{sec:con}

\section{Acknowledgment}
This research was supported by the  Austrian Science Fund, FWF Project (V312-N25). 

%%%%%%%%%%%%%%%%%%%%%%%%%%%%%%%%%%%%%%%%%%%%%%%%%%%%%%%%%%%%

\bibliographystyle{abbrv}
%\bibliography{model_sets}

\begin{thebibliography}{10}

\bibitem{BM00}
M.~{Baake} and R.~V. {Moody}.
\newblock {Self-similar measures for quasicrystals.}
\newblock In {\em {Directions in mathematical quasicrystals}}, pages 1--42.
  Providence, RI: AMS, American Mathematical Society, 2000.

\bibitem{Bes55}
A.~{Besicovitch}.
\newblock {Almost periodic functions.}
\newblock {Neudruck. New York: Dover Publications, Inc. XIII, 180 p. (1955).},
  1955.

\bibitem{BB31}
A.~{Besicovitch} and H.~{Bohr}.
\newblock {Almost periodicity and general trigonometric series.}
\newblock {\em {Acta Math.}}, 57:203--292, 1931.

\bibitem{BFG17}
P.~{Boggiatto}, C.~{Fern\'andez}, and A.~{Galbis}.
\newblock {Gabor systems and almost periodic functions.}
\newblock {\em {Appl. Comput. Harmon. Anal.}}, 42(1):65--87, 2017.

\bibitem{B47}
H.~{Bohr}.
\newblock {\em {Almost periodic functions.}}
\newblock Chelsea Publishing Company, 1947.

\bibitem{BDR94}
C.~{de Boor}, R.~A. {DeVore}, and A.~{Ron}.
\newblock {The structure of finitely generated shift-invariant spaces in $L\sb
  2(\mathbb R\sp d)$.}
\newblock {\em {J. Funct. Anal.}}, 119(1):37--78, 1994.

\bibitem{Fei81}
H.~G. {Feichtinger}.
\newblock {On a new Segal algebra.}
\newblock {\em {Monatsh. Math.}}, 92:269--289, 1981.

\bibitem{FL06}
H.~G. {F}eichtinger and F.~{L}uef.
\newblock {W}iener amalgam spaces for the {F}undamental {I}dentity of {G}abor
  {A}nalysis.
\newblock {\em {C}ollect. {M}ath.}, 57({E}xtra {V}olume (2006)):233--253, 2006.

\bibitem{Ga09}
J.-P. {G}abardo.
\newblock {W}eighted irregular {G}abor tight frames and dual systems using
  windows in the {S}chwartz class.
\newblock {\em {J}. {F}unct. {A}nal.}, 256(3):635 -- 672, 2009.

\bibitem{G04}
F.~{Galindo}.
\newblock {Some remarks on ``On the windowed Fourier transform and wavelet
  transform of almost periodic functions,'' by J. R. Partington and B.
  \"Unalm{\i}\c s.}
\newblock {\em {Appl. Comput. Harmon. Anal.}}, 16(3):174--181, 2004.

\bibitem{gr01}
K.~{G}r{\"o}chenig.
\newblock {\em {F}oundations of {T}ime-{F}requency {A}nalysis}.
\newblock {A}ppl. {N}umer. {H}armon. {A}nal. {B}irkh{\"a}user {B}oston, 2001.

\bibitem{GOR15}
K.~{Gr\"ochenig}, J.~{Ortega-Cerd\`a}, and J.~L. {Romero}.
\newblock {Deformation of Gabor systems.}
\newblock {\em {Adv. Math.}}, 277:388--425, 2015.

\bibitem{grrost17}
K.~{G}r{\"o}chenig, J.~L. {R}omero, and J.~{S}t{\"o}ckler.
\newblock {S}ampling theorems for shift-invariant spaces, {G}abor frames, and
  totally positive functions.
\newblock {\em to appear in {I}nvent. {M}ath.}, pages 1--30, 2017.

\bibitem{HLW02}
E.~{Hern\'andez}, D.~{Labate}, and G.~{Weiss}.
\newblock {A unified characterization of reproducing systems generated by a
  finite family. II.}
\newblock {\em {J. Geom. Anal.}}, 12(4):615--662, 2002.

\bibitem{HW96}
E.~{Hern\'andez} and G.~{Weiss}.
\newblock {\em {A first course on wavelets.}}
\newblock Boca Raton, FL: CRC Press, 1996.

\bibitem{JaLe16}
M.~{J}akobsen and J.~{L}emvig.
\newblock {C}o-compact {G}abor systems on locally compact {A}belian groups.
\newblock {\em {J}. {F}ourier {A}nal. {A}ppl.}, 22(1):36--70, 2016.

\bibitem{J95}
A.~{Janssen}.
\newblock {Duality and biorthogonality for Weyl-Heisenberg frames.}
\newblock {\em {J. Fourier Anal. Appl.}}, 1(4):403--436, 1995.

\bibitem{K16}
M.~{Kreisel}.
\newblock {Gabor frames for quasicrystals, $K$-theory, and twisted gap
  labeling.}
\newblock {\em {J. Funct. Anal.}}, 270(3):1001--1030, 2016.

\bibitem{L02}
D.~{Labate}.
\newblock {A unified characterization of reproducing systems generated by a
  finite family.}
\newblock {\em {J. Geom. Anal.}}, 12(3):469--491, 2002.

\bibitem{MaMe10}
B.~{M}atei and Y.~{M}eyer.
\newblock Simple quasicrystals are sets of stable sampling.
\newblock {\em Complex Var. Elliptic Equ.}, 55(8-10):947–964, 2010.

\bibitem{M18}
E.~{Matusiak}.
\newblock {Frames of translates for model sets.}
\newblock {\em {preprint}}, 2018.

\bibitem{Me72}
Y.~{Meyer}.
\newblock {Algebraic numbers and harmonic analysis.}
\newblock {North-Holland Mathematical Library. Vol. 2. Amsterdam-London: North-
  Holland Publishing Co mpany. X, 274 p.}, 1972.

\bibitem{Me12}
Y.~{Meyer}.
\newblock {Quasicrystals, almost periodic patterns, mean-periodic functions and
  irregular sampling.}
\newblock {\em {Afr. Diaspora J. Math.}}, 13(1):1--45, 2012.

\bibitem{MNP08}
R.~{Moody}, M.~{Nesterenko}, and J.~{Patera}.
\newblock {Computing with almost periodic functions.}
\newblock {\em {Acta Crystallogr., Sect. A}}, 64(6):654--669, 2008.

\bibitem{M05}
R.~V. {Moody}.
\newblock {Mathematical quasicrystals: a tale of two topologies.}
\newblock In {\em {XIVth international congress on mathematical physics (ICMP
  2003), Lisbon, Portugal, 28 July -- 2 August 2003. Selected papers based on
  the presentation at the conference.}}, pages 68--77. Hackensack, NJ: World
  Scientific, 2005.

\bibitem{PU01}
J.~{Partington} and B.~{\"Unalm{\i}\c{s}}.
\newblock {On the windowed Fourier transform and wavelet transform of almost
  periodic functions.}
\newblock {\em {Appl. Comput. Harmon. Anal.}}, 10(1):45--60, 2001.

\bibitem{Sch00}
M.~{Schlottmann}.
\newblock {Generalized model sets and dynamical systems.}
\newblock In {\em {Directions in mathematical quasicrystals}}, pages 143--159.
  Providence, RI: AMS, American Mathematical Society, 2000.

\bibitem{WR90}
J.~{W}exler and S.~{R}az.
\newblock {D}iscrete {G}abor expansions.
\newblock {\em {S}ignal {P}rocess.}, 21:207--220, 1990.

\end{thebibliography}

\end{document}